\documentclass[12pt]{article}
\usepackage{amsfonts,amssymb,amsbsy,amsmath,amsthm,latexsym,natbib,graphicx}
\usepackage{psfrag,epsfig,epsf,dsfont} 
\usepackage{subfigure}
\usepackage[margin=1in]{geometry}

\newtheorem{prop}{Proposition}
\newtheorem{theorem}{Theorem}
\newtheorem{lemma}{Lemma}
\newtheorem{corollary}{Corollary}

\newcommand{\Prob}[1]{\mathbb{P}\left({#1}\right)}

\newcommand{\Expect}[1]{\mathbb{E}\left[{#1}\right]}
\newcommand{\Expects}[2]{\mathbb{E}_{{#1}}\left[{#2}\right]}
\newcommand{\Var}[1]{\mathsf{Var}\left[{#1}\right]}
\newcommand{\Vars}[2]{\mathsf{Var}_{#1}\left[{#2}\right]}
\newcommand{\Cor}[1]{\mathsf{Corr}\left[{#1}\right]}

\newcommand{\asVar}{\mathsf{Var}}
\newcommand{\PMH}{\mathsf{P}_{{\mathrm{MH}}}}
\newcommand{\Ppm}{\mathsf{P}_{{\mathrm{pm}}}}
\newcommand{\Pcpm}{\mathsf{P}_{{\mathrm{cpm}}}}
\newcommand{\Ptilpm}{\widetilde{\mathsf{P}}_{{\mathrm{pm}}}}

\newcommand{\epsMH}{\epsilon_{{\mathrm{MH}}}}
\newcommand{\leftGapPM}{\epsilon^{\mathrm{L}}_{\mathrm{pm}}}

\newcommand{\pihat}{\widehat{\pi}}
\newcommand{\hhat}{\widehat{h}}

\newcommand{\Phat}{\widehat{P}}
\newcommand{\thetahat}{\widehat{\theta}}

\newcommand{\qtil}{\widetilde{q}}

\newcommand{\alphabar}{\overline{\alpha}}
\newcommand{\alphatil}{\widetilde{\alpha}}

\newcommand{\Ptil}{\widetilde{P}}
\newcommand{\Qtil}{\widetilde{Q}}
\newcommand{\md}{\mathsf{d}}
\newcommand{\cE}{\mathcal{E}}
\newcommand{\cX}{\mathcal{X}}
\newcommand{\cW}{\mathcal{W}}

\newcommand{\ESS}{\mathsf{ESS}}

\graphicspath{{.}}

\theoremstyle{plain}
\newtheorem{assumption}{Assumption}

\title{Variance bounds and robust tuning for pseudo-marginal Metropolis--Hastings algorithms}
\author{Chris Sherlock}
\date{}
\begin{document}

\maketitle
\vspace{-1.5cm}

\begin{abstract}
  The general applicability and ease of use of the pseudo-marginal Metropolis--Hastings (PMMH) algorithm, and particle Metropolis--Hastings in particular, makes it a popular method for inference on discretely observed Markovian stochastic processes. The performance of these algorithms and, in the case of particle Metropolis--Hastings, the trade off between improved mixing through increased accuracy of the estimator and the computational cost were investigated independently in two papers, both published in 2015. Each suggested choosing the number of particles so that the variance of the logarithm of the estimator of the posterior at a fixed sensible parameter value is approximately 1. This advice
  has been widely and successfully adopted.
  We provide new, remarkably simple upper and lower bounds on the asymptotic variance of PMMH algorithms. The bounds explain how blindly following the 2015 advice can hide serious issues with the algorithm and they strongly suggest an alternative criterion. In most situations our guidelines and those from 2015 closely coincide; however, when the two differ it is safer to follow the new guidance. An extension of one of our bounds shows how the use of correlated proposals can fundamentally shift the properties of pseudo-marginal algorithms, so that asymptotic variances that were infinite under the PMMH kernel become finite.
  \end{abstract}

\section{Introduction}
\label{sec.intro}
The Metropolis-Hastings Markov chain Monte Carlo algorithm \cite[e.g.,][]{MCMChandbook} is a frequent method of choice for Bayesian inference on the parameters of a statistical model when it is straightforward to evaluate the posterior pointwise up to a normalisation constant. In many common scenarios, however, pointwise evaluation of the likelihood is not feasible; for example, because it involves a high-dimensional integral over latent variables. In such cases, bespoke solutions were traditionally obtained by extending the statespace of the Metropolis--Hastings chain to include the latent variables. More recently, however, the pseudo-marginal Metropolis--Hastings algorithm \cite[PMMH,][]{AndRob2009} has offered a convenient, off-the-shelf alternative: the unobtainable likelihood is replaced with a realisation of an unbiased estimator of it, and the statespace of the Markov chain is extended by a single scalar quantity: that estimator. 

The simplicity and power of PMMH and, in particular, of a version that we will call particle MH, which is suitable for hidden Markov models and obtains the unbiased estimate through a particle filter \cite[]{ADH2010}, have ensured its popularity in practice. At the time of writing, Scopus reports over $500$ citations for \cite{AndRob2009} and over $1300$ for \cite{ADH2010}.

Alongside the extensive usage of these algorithms comes a pressing need to understand their properties: when they can be trusted to behave well and how to tune them to balance the trade-off between computational cost per iteration, often specified through a number of samples or particles, and the asymptotic variance of the estimators of posterior expectations. This trade off was investigated in \cite{PSGK2012}, then \cite{DPDK2015} and \cite{STRR2015}. These papers use different methods, with the first two bounding the asymptotic variance and the last exploring the behaviour of a diffusion limit, but the tuning advice is very similar: choose a number of particles such that (at a sensible choice of the parameter) the variance of the logarithm of the estimator of the posterior is approximately $1$; \cite{She2016} demonstrates the robustness of this tuning advice to the choice of another tuning parameter. Between them, these papers have over $400$ citations on Scopus, demonstrating the uptake of this tuning advice.

The (non-geometric) convergence properties were first investigated in \cite{AndVih2015}, with more recent work in \cite{ALPW2022}. Finally, \cite{DelLee2018} examines necessary and sufficient conditions for the asymptotic variance of an estimator of a particular expectation to be finite. Of particular relevance here is that in all three of these articles, sufficient conditions for good behaviour (whether relatively fast convergence or finite asymptotic variance) were tied to the finiteness of polynomial moments of the estimator of the likelihood.

We provide three main contributions:
\begin{enumerate}
\item We derive simple, explicit upper and lower bounds on the asymptotic variance of estimators of posterior expectations that are obtained via PMMH. The bounds make it crystal clear that it is the second moment of the estimator of the likelihood that is key to good behaviour, not the second moment of its logarithm. Not only are our bounds simpler to state than those in \cite{DPDK2015} but they are also more straightforward to derive. We demonstrate that, despite this, our first bound is almost indistinguishable from the corresponding bound in \cite{DPDK2015}.
\item We correct the tuning advice from \cite{PSGK2012}, \cite{DPDK2015} and \cite{STRR2015}. The number of particles should not be chosen according to the variance of the logarithm of the likelihood estimator but according to the relative variance of the estimator. By the delta method, when this variance is relatively small, these criteria are almost equivalent; however, when the variance is large our advice will protect against poor performance when it is possible to do so, and it is likely to alert the user to more fundamental issues, such as infinite asymptotic variance of resulting estimators.
\item We both prove and demonstrate through simulation that a PMMH algorithm with infinite asymptotic variances may be rescued, so that it has finite asymptotic variances, through the use of the correlated pseudo-marginal Metropolis--Hastings algorithm \cite[]{DLKS2015,DDP2018}. \cite{DDP2018} establishes that correlated PMMH can achieve the same performance as PMMH using a reduced number of particles compared with PMMH; however, our result shows that correlated PMMH can fundamentally shift the properties of the algorthm.
\end{enumerate}

\section{Background}

\subsection{Kernels and acceptance probabilities}
Consider a target posterior distribution $\pi$ on $\cX\subseteq \mathbb{R}^d$ with a density of $\pi(\theta)$ with respect to Lebesgue measure. Let the estimators of the posterior (up to a normalisation constant) at $\theta$ and $\theta'$ be $\pihat(\theta;U)$ and $\pihat(\theta';U')$, where $U$ and $U'$ are auxiliary variables
sampled from densities of $q_*(u|\theta)$ and $q_*(u'|\theta')$ respectively.
We define the multiplicative noises as
\begin{equation}
  \label{eqn.multnoise}
  W:=\frac{\pihat(\theta;U)}{\pi(\theta)}
  ~~~\mbox{and}~~~
    W':=\frac{\pihat(\theta';U')}{\pi(\theta')}.
\end{equation}
The auxiliary sampling density,  $q_*(u|\theta)$, implies a proposal density for $w$, $\qtil_\theta(w)$; we also define $\nu_\theta(w):=w\qtil_{\theta}(w)$. The extended statespace of the pseudo-marginal chain is $\cX\times \cW$, where $\cW\subseteq [0,\infty)$.

Our interest lies in the following two kernels:
\begin{itemize}
\item $\PMH$ is a Metropolis-Hastings kernel on $\cX$ with a stationary density of $\pi(\theta)$. It proposes from $q(\theta|\theta')$ and accepts with a probability of
  \begin{equation}
    \label{eqn.alphaMH}
    \alpha_{MH}(\theta,\theta'):=1\wedge r(\theta,\theta'),~~~\mbox{where}~r(\theta,\theta'):=\frac{\pi(\theta')q(\theta|\theta')}{\pi(\theta)q(\theta'|\theta)}.
  \end{equation}
\item $\Ppm$ is the corresponding pseudo-marginal MH kernel on $\cX\times \cW$  with a stationary density of $\pi(\theta)\nu_{\theta}(w)$. This is a density since the noise being unbiased implies $\int_{\cW}w\qtil_\theta(w)\md w = \Expects{\qtil_\theta}{W}=1$. It proposes from $q(\theta'|\theta)\qtil_{\theta}(w')$ and accepts with a probability of
  \begin{equation}
    \label{eqn.alphaPM}
  \alpha_{pm}(\theta,w'\theta',w'):=1\wedge \frac{\pihat(\theta';u')q(\theta|\theta')}{\pihat(\theta;u)q(\theta'|\theta)}\equiv 1\wedge \left\{r(\theta,\theta')\frac{w'}{w}\right\}.
\end{equation}
\end{itemize}
The stationary density of $\Ppm$ can be seen to be $\pi(\theta)\nu_{\theta}(w)$ since
\begin{equation}
  \label{eqn.PMDB}
\pi(\theta)\nu_\theta(w)q(\theta'|\theta)\qtil_{\theta'}(w')\alpha_{pm}(\theta,w;\theta',w')=\pi(\theta')\nu_{\theta'}(w')q(\theta|\theta')\qtil_{\theta}(w)\alpha_{pm}(\theta',w';\theta,w).
\end{equation}

In proving some of our results, we will makes use of a third kernel:
\begin{itemize}
\item $\Ptilpm$ is a handicapped pseudo-marginal kernel on $\cX\times \cW$ that also has a stationary density of $\pi(\theta)\nu_\theta(w)$. It proposes from $q(\theta'|\theta)\qtil_{\theta'}(w')$ and accepts with a probability of
  \begin{equation}
    \label{eqn.alphaPMhandi}
  \alphatil_{pm}(\theta,w;\theta',w'):=\left\{1\wedge \frac{w'}{w}\right\}\times \{1\wedge r(\theta,\theta')\}.
  \end{equation}
  $\Ptilpm$ has the same stationary distribution as $\Ppm$ since \eqref{eqn.PMDB} continues to hold with $\alpha_{pm}$ replaced by $\alphatil_{pm}$. 
\end{itemize}

\subsection{Asymptotic variance}

Consider a Markov transition kernel, $Q$, on a statespace $\cX$ and with a stationary distribution of $\mu$. For a function of interest $h(\cdot)$, after $n$ iterations, the ergodic average used to estimate $\Expects{\mu}{h}$ is $\hhat_n:=\frac{1}{n}\sum_{i=1}^n h(x_i)$, where $x_i$ is the value of the chain after $i$ iterations. One of the most natural measures of the inefficiency of the chain is the asymptotic variance of this average:
\begin{equation}
  \label{eqn.defineAsympVar}
\asVar_Q(h):=\lim_{n\to \infty} n \Var{\frac{1}{n}\sum_{i=1}^n h(X_i)},~~~X_1\sim \mu.
\end{equation}
If $\asVar_Q(h)<\infty$ then for large $n$, $\Var{\hhat_n}\approx \asVar_Q(h)/n$. Since the bias decreases in proportion to $1/n$ the typical error in $\hhat_n$ decreases as $1/\sqrt{n}$; furthermore \cite[e.g.,][]{Gey1992}, assuming the chain is ireducible and aperiodic, $\hhat_n$ satisfies a central limit theorem: $\sqrt{n}(\hhat_n-\Expects{\mu}{h})\Rightarrow \mathsf{N}\left(0,\asVar_Q(h)\right)$.

When $\asVar_Q(h)<\infty$, since doubling the length of the chain halves the variance but has twice the computational cost, a natural measure of the computational inefficiency of $Q$ for estimating $\Expects{\mu}{h}$ is $C_Q\asVar_Q(h)$, where $C_Q$ is the computational cost per iteration of $Q$. In terms of optimising efficiency, $C_Q \asVar_Q(h)$ is the quantity of prime interest. However, this measure is meaningless when $\asVar_Q(h)=\infty$. 

Since $\alphatil_{pm}(\theta,w;\theta',w')\le \alpha_{pm}(\theta,w;\theta',w')$ it follows \cite[]{Tie1998} that  $\asVar_f(\Ptilpm)\ge \asVar_f(\Ppm)$ for all $f\in L^2(\pi\times \nu)$. Hence an upper bound on the asymptotic variance of the handicapped kernel is also an upper bound on the asymptotic variance of the kernel of interest.

The Dirichlet form of a function $f\in L^2(\mu)$ is
\[
\cE_Q(f)=\frac{1}{2}\iint_{\cX\times \cX} \mu(\md x) Q(x,\md y)\left\{f(y)-f(x)\right\}^2.
\]
The (right) spectral gap of $Q$ is then 
\[
\epsilon_{Q}=\inf_{f\in L^2_0(\mu): \|f\|_{L^2(\mu)}= 1} \cE_Q(f).
\]
Considering $\|f\|_{L^2(\mu)}=1$ with $f\in L_0^2(\mu)$, the above two equations give
\[
1-\iint_{\cX\times \cX}\mu(\md x) Q(x, \md y)f(x)f(y)\ge \epsilon_Q.
\]
That is, the largest possible lag-$1$ auto-correlation for any $f\in L^2(\mu)$ is $\rho=1-\epsilon_Q$. For any $f\in L^2(\mu)$ with $\Vars{\mu}{f}=1$, \cite[e.g.,][]{Gey1992},
\begin{equation}
  \label{eqn.stdSpecGapBd}
\asVar_Q(f)\le \frac{1+\rho}{1-\rho}=\frac{2-\epsilon_Q}{\epsilon_Q}.
\end{equation}
Because of this, any kernel $Q$ with a non-zero right spectral gap is termed \emph{variance bounding} \cite[]{RobRos2008}. Our upper bounds will arise from the following variational representation of the asymptotic variance \cite[e.g.,][proof of Lemma 33]{ALV2018}:
\begin{equation}
  \label{eqn.varRepvar}
  \asVar_Q(f)
  =
  \sup_{g\in L_0^2(\mu)} 4 \langle f,g\rangle - 2 \cE_Q(g)-\langle f,f \rangle.
\end{equation}
As well as the right spectral gap, we will also require the left spectral gap. For our kernel $Q$, this is:
\[
\epsilon_Q^L:=1+\inf_{f\in L_0^2(\mu), \|f\|=1}\iint_{\cX\times \cX} \mu(\md x)Q(x,\md y) f(x)f(y).
\]

\subsection{Literature and common assumptions}
The following assumption is common in earlier analyses and we, too, will be making it for some of our work:
\begin{assumption}
  \label{ass.indepNoise}
The proposed value of the multiplicative noise, $W'$, has a density or mass function of $\qtil(w')$ that is independent of $\theta'$.
\end{assumption}
This is justified in the large-data regime by the posterior mass concentrating around the true parameter value according to the Bernstein-von Mises Theorem and by the intuition that small changes in the parameter value will lead to small changes in the properties of the log-likelihood estimator from the particle filter; see \cite{SDDP2021}. Under Assumption \ref{ass.indepNoise}, the stationary density/mass function of $W$ simplifies to $\nu(w)=w\qtil(w)$.

All three of \cite{PSGK2012}, \cite{DPDK2015} and \cite{STRR2015} make the following assumption when moving from general bounds or limits to specific tuning guidance:
\begin{assumption}
  \label{ass.lognormalCLT}
The proposed value of the multiplicative noise, $W'$, satisfies $\log W'\sim \mathsf{N}(-\frac{1}{2}\sigma^2,\sigma^2)$ with the computational cost $\propto 1/\sigma^2$.
\end{assumption}
For a particle filter, as the number of observations $T$ increases to $\infty$ with the number of particles $N\propto T$,  \cite{BDD2014} shows that, subject to conditions, such a log-normal central limit theorem holds. Furthermore, the variance, $\sigma^2\propto 1/N$, that is, the variance is inversely proportional to the computational cost.

As in \cite{PSGK2012} and \cite{DPDK2015}, our target is the asymptotic variance. Both of these articles make Assumption \ref{ass.indepNoise} and analyse $\Ptilpm$ rather than $\Ppm$. \cite{DPDK2015} extends \cite{PSGK2012} from the independence sampler to a general Metropolis--Hastings kernel.  To do this
it considers the \emph{jump chain} which is the set of accepted values and the number of iterations the chain stays at that value: $\{\theta_*^i,\tau^i\}_{i=1}^\infty$. The asymptotic variance of the jump chain for $\Ptilpm$ is related to that of $\Ptilpm$ itself, giving an intractable expression for the latter. Bounds on this expression are provided, and in the special case of Assumption \ref{ass.lognormalCLT} these require only two numerical integrations in order to evaluate. The more generally applicable upper bound is (for $h$ with $\Vars{\pi}{h}=1$):
\[
1+\asVar_{\Ppm}(h)\le
\{1+\asVar_{\PMH}(h)\}
\left[\Expects{\nu}{1/\alphabar_w(W)}+
  (1-\phi_*)\{\Expects{\nu}{1/\alphabar_w(W)}-1/\Expects{\nu}{\alphabar_w(W)}\}
  \right],
\]
where $\alphabar_w(W)=\int_{\cW}\qtil(w') \{1\wedge w'/W\}\md w'$ and $\phi_*$ is the lag-1 autocorrelation of $1/\alphabar_w(W)$ under the jump chain whose next move is to $w'$ with a density of $\qtil(w')\{1\wedge w'/w\}/\alphabar_w(w)$. The minimiser, $\sigma_{opt}$, of the computational inefficiency of the bound (the bound multiplied by $\sigma^2$) depends on $\asVar_{\PMH}(h)$. When $\asVar_{\PMH}(h)=1$, $\sigma_{opt}\approx 0.92$ and as $\asVar_{\PMH}(h)\uparrow \infty$, $\sigma_{opt}\to 1.68$. A value of $\sigma_{opt}\approx 1.2$ is found to minimise the maximum penalty across all values for $\asVar_{\PMH}(h)$. 

\cite{DelLee2018} investigates which functions of the pseudo-marginal chain have a finite asymptotic variance. As with Deligiannidis et al (2015), this is approached via the jump chain. The following quantity is central:
\begin{equation}
  \label{eqn.definesbar}
  \overline{s}=\pi_{\mathsf{ess} \sup} \int_{\cW} w^2 q_\theta(w) \md w.
\end{equation}
If $\overline{s}<\infty$ and $\PMH$ is variance bounding then $f\in L^2(\pi)$ has a finite asymptotic variance. In the special case of Assumption \ref{ass.indepNoise}, to achieve finite $\asVar_{\Ppm}(f)$, it is found that only the jump chain need be variance bounding as long as $\asVar_{\PMH}(f)<\infty$, too.

\subsection{Examples}

We provide two simple motivating examples. The first makes it clear that tuning according to  $\Var{\log W}$ cannot be the correct general advice since this quantity is not even defined. The second has $\Var{\log W}<\infty$ even though $\Var{W}=\infty$ and (hence) the asymptotic variance is infinite; it will be revisted in Sections \ref{sec.corrPMMH} and in \ref{sec.num.heavy} where the use of a correlated pseudo-marginal algorithm leads to a finite asymptotic variance.

\subsubsection{Complete exact observations}
\label{sec.exactObs}
Consider the simplest case of complete, exact observations, and where for each inter-observation interval we use $n$ samples. For the $t$th interval, let $p_t$ be the probability of success (hitting the observation) and let $S_t$ be the number of successes. Then the estimator of the $t$th transition probability is $\Phat_t=S_t/n$, and $\Phat=\prod_{t=1}^T \Phat_t$ is an unbiased estimator of the likelihood. 
For each $t$, $\Prob{\Phat_t=0}=(1-p_t)^n$, so
$\Prob{\Phat=0}=1-\prod_{t=1}^T\left\{1-(1-p_t)^n\right\}>0$.
Thus $\Expect{\log \Phat}=-\infty$ and $\Var{\log \Phat}$ is undefined.

Similar examples that use particle filters and yet have $\Var{\log \Phat}$ undefined include particle filters for partial but precise observations and ABC-particle MCMC, where at time $t$, particles are kept if their summary statistics fall within some $\delta$ of the summary statistic of the observation. 

\subsubsection{A shifted Pareto distribution}
\label{sec.shiftedPareto}
Let
\begin{equation}
  \label{eqn.heavy}
\qtil_W(w;a)=\frac{a}{(1+w)^{1+a}}~~~(w\ge 0,~a>1).
\end{equation}
  For $a>1$, this has $\Expects{\qtil}{W}=1/(a-1)<\infty$ but for $a\le 2$, $\Vars{\qtil}{W}=\infty$.
  However $V:=\log W$ has a density of
  \[
f_V(v)=\frac{a\exp(v)}{\{1+\exp(v)\}^{1+a}},
\]
so all polynomial moments of $\log W$ are finite. In particular, when $a=2$, $\Expect{W}=1$ and numerical integration gives $\Var{\log W}\approx 2.29$. So, for two independent realisations, $W_1$ and $W_2$ from this distribution, $\Var{\log\{(W_1+W_2)/2\}}\approx 1.14$, fitting with the current tuning advice.

\section{Bounds on the asymptotic variance}
\label{sec.VarBounds}

This section provides quantitative upper and lower bounds on the asymptotic variance of $\Ppm$ which make clear that in many circumstances, subject to Assumption \ref{ass.indepNoise}, $\Expects{\qtil}{W^2}<\infty$ is both a necessary and sufficient condition for a finite asymptotic variance of ergodic averages of $h\in L^2(\pi\times \nu)$. Indeed, Theorem \ref{thrm.necessary} makes clear that the behaviour of $W^2$ is important even without Assumption \ref{ass.indepNoise}.

\subsection{Upper and lower bounds on the asymptotic variance}
\label{sec.TheBounds}
Our central upper bound is:
\begin{theorem}
\label{theorem.varboundA}
If $\PMH$ has a right-spectral gap of $\epsMH$ and Assumption \ref{ass.indepNoise} holds for $\Ppm$ then for any $h_*\in L^2(\pi)$, $\asVar_{\Ppm}(h_*) \le \{2R_S^{\qtil}/\epsMH-1\} \|h_*\|_{L^2(\pi)}$, where 
  \begin{equation}
    \label{eqn.aVarUB}
 R_S^{\qtil} :=
  \iint_{\cW\times \cW} \qtil(w)\qtil(w') ww'(w\vee w')\md w \md w'.
\end{equation}
\end{theorem}

We will generalise the bound of Theorem \ref{theorem.varboundA} in several different different directions, with proofs placed in the appendices. However, since the fundamental argument is straightforward and illuminates the key ideas, we prove Theorem \ref{theorem.varboundA} itself in Section \ref{sec.proveMainTheorem} of the main text.

Firstly, Theorem \ref{theorem.varboundA} is as tight as could be hoped for when there is no multiplicative noise; \emph{i.e.} when $\nu(w)\equiv \qtil(w)=\delta(w-1)$, $R^{\qtil}_1=1$ and we obtain the standard spectral gap bound \eqref{eqn.stdSpecGapBd}. Secondly, since $(w+w')/2\le w\vee w'\le w+w'$
\begin{equation}
  \label{eqn.tractableGen}
  \int_{\cW}\qtil(w) w^2 \md w \le R^{\qtil}_1
  \le 2\int_{\cW}\qtil(w) w^2 \md w,
\end{equation}
so $R^{\qtil}_1<\infty \Leftrightarrow \int_{\cW}\qtil(w) w^2 \md w<\infty$. In particular, if $\PMH$ is variance bounding then subject to Assumption \ref{ass.indepNoise}, $\Expects{\qtil}{W^2}<\infty$ is a sufficient condition for $\asVar_{\Ppm}(h_*)$ to be variance bounding for all $h_* \in L^2(\pi)$.

Proposition 4 of \cite{DelLee2018} also shows that $\asVar_{\Ppm}<\infty$ requires $\Expects{\qtil}{W^2}<\infty$ but it needs only that the jump chain of $\PMH$ be variance bounding and that $\asVar_{\Ppm}(h_*)<\infty$. The novel contribution from Theorem \ref{theorem.varboundA} is the simple, explicit upper bound that makes the importance of $\Expects{\qtil}{W^2}$ crystal clear.

Given Assumption \ref{ass.indepNoise} and a right spectral gap for the Metropolis--Hastings kernel, $\Expects{\qtil}{W^2}<\infty$ is sufficient for variance bounding, but perhaps a weaker condition would suffice? We now provide an explicit lower bound, a corollary of which is that $\Expects{\qtil}{W^2}<\infty$ is often also necessary. 

We require the following quotient of the marginal proposal density when the chain is stationary and the stationary density:
\begin{align}
  \label{eqn.define.rbar}
    r(\theta)&:=\Expects{\theta'\sim q(\cdot|\theta)}{r(\theta,\theta')}
    =
\frac{1}{\pi(\theta)}\int \pi(\theta')q(\theta|\theta') \md \theta'.
\end{align}
We also require the existence of a left-spectral gap for $\Ppm$: $\leftGapPM>0$.

This ensures that the pseudo-marginal kernel has no fully periodic component. It holds automatically, with $\leftGapPM=1$, for any positive kernel (that is, kernel with a non-negative spectrum) such as for pseudo-marginal algorithms based on the independence sampler or the random-walk Metropolis with a Gaussian proposal. Any other kernel can be modified to have a spectral gap of at least $\delta<1$ by setting $P^{\mathrm{lazy}}_{\mathrm{pm}}=\delta \mathrm{Id} + (1-\delta) \Ppm$.

Theorem \ref{thrm.necessary} is proved in Appendix \ref{sec.prove.Necessary}. The central argument is that in the presence of a left spectral gap the contribution of auto-correlations with odd-numbered lags is bounded below; auto-correlations with  even-numbered lags are lower bounded by a function of the probability that the chain still has not moved from its starting position by that lag. The theorem does \emph{not} require Assumption \ref{ass.indepNoise}.

\begin{theorem}
  \label{thrm.necessary}
  Let $r(\theta)$ be as defined in \eqref{eqn.define.rbar}. If $\Ppm$ has a left spectral gap of $\leftGapPM$ then for all $h\in L^2(\pi\times \nu)$,
  \[
  \asVar_{\Ppm}(h)
  \ge
\left\{-\frac{1}{\leftGapPM}+\frac{\leftGapPM}{4-2\leftGapPM}\right\}\|h\|^2_{L^2(\pi\times \nu)}+  \frac{1}{2}\int \pi(\theta) \qtil_{\theta}(w) h(\theta,w)^2 \frac{w^2}{r(\theta)}\md w \md \theta.
\]
\end{theorem}

\begin{corollary}
Suppose that Assumption \ref{ass.indepNoise} holds, that $\Ppm$ has a left-spectral gap of $\leftGapPM$, that  $h(\theta,w)=h_*(\theta)\in L^2(\pi)$  and for $\pi$-almost all $\theta$, $r(\theta)\le c$ for some $c<\infty$. Then
\[
  \asVar_{\Ppm}(h)
  \ge
-\frac{1}{\leftGapPM}\|h_*\|^2_{L^2(\pi)}
  +\frac{1}{2c}\|h_*\|^2_{L^2(\pi)} \int \qtil(w) w^2 \md w.
  \]
  So $\Expects{\qtil}{W^2}=\infty\implies \asVar_{\Ppm}(h)=\infty$.
\end{corollary}  

The condition that $r(\theta)\le c<\infty$ is mild. For example, in the case of the random-walk Metropolis with $\theta'\sim \mathsf{N}(\theta,\lambda^2 I_d)$, it is straightforward to show that it holds provided the most positive eigenvalue of the Hessian of $\log \pi(\theta)$ is ($\pi$-almost everywhere) less than $1/\lambda^2 - \epsilon$ for some $\epsilon>0$. This is guaranteed, for example, if $\pi$ is log-concave.  

\cite{DelLee2018} examines the Metropolis-Hastings independence sampler (MHIS) and, in its Proposition 3, gives the necessary and sufficient condition for a finite asymptotic variance of $h$ to be:
\[
\int_{\cX} \pi(\theta)\qtil_\theta(w) w^2\frac{\pi(\theta)}{q(\theta)}h(\theta,w)^2~\md \theta \md w<\infty.
\]
For the MHIS, $r(\theta)=q(\theta)/\pi(\theta)$, so the (necessary part of) the condition is equivalent to ours. If $\mathsf{P}_{\mathrm{MHIS}}$  is a geometrically ergodic independence sampler then $\pi(\theta)/q(\theta)>c$ for some $c>0$ \cite[e.g.][]{RobRos2011}. Thus, if $\qtil_{\theta}=\qtil$ and $h(\theta,w)=h_*(\theta)$, we again find that $\asVar_{\Ppm}=\infty$ if either $\int_\cW w^2 \qtil(w) \md w=\infty$ or $\|h_*\|_{L^2(\pi)}=\infty$; see also \cite{DelLee2018}, Corollary 3.

\subsection{Proof of Theorem \ref{theorem.varboundA}}
\label{sec.proveMainTheorem}
The proof proceeds in three stages: first we lower bound the Dirichlet form for $g$; second we upper bound $\langle h_*,g\rangle_{L^2(\pi\times \nu)}$; finally we combine the two through \eqref{eqn.varRepvar}. Unless explicitly stated otherwise, all integrals with respect to $\theta$ and/or $\theta'$ are over $\cX$ and all integrals with respect to $w$ and/or $w'$ are over $\cW$.

All stages use the following natural construction: for any function $g(\theta,w)\in L^2(\pi\times \nu)$ and for each $w \in \cW$ we identify the function $g_w:\cX\to \mathbb{R}$ as
\begin{equation}
  \label{eqn.define.gw}
  g_w(\theta)=g(\theta,w),
\end{equation}
We note the following natural inheritance (proof in Appendix \ref{sec.proof.prop.LsqInherit} for completeness):

\begin{prop}
  \label{prop.LsqInherit}
For every $g(\theta,w)\in L^2(\pi\times \nu)$,  $\nu$ almost every $g_w(\theta)\in L^2(\pi)$, where $g_w$ is defined as in \eqref{eqn.define.gw}. 
\end{prop}
    
\textbf{Stage 1}: We first bound the Dirichlet form for $g$. Recalling that $\nu(w)=w\qtil(w)$, this is
\[
  \cE_{\Ppm}(g)=
  \frac{1}{2}\int \pi(\theta) w\qtil(w) q(\theta'|\theta)\qtil(w')
  \alpha_{pm}(\theta,w;\theta',w') \{g(\theta',w')-g(\theta,w)\}^2\md \theta \md \theta' \md w \md w'.
\]
Since $\alpha_{pm}(\theta,w;\theta',w')\ge\alphatil_{pm}(\theta,w;\theta',w') \ge \{1\wedge w'/w\}\{1\wedge r(\theta,\theta')\}\equiv \{1\wedge w'/w\}\alpha_{MH}(\theta,\theta')$,
\begin{align}
  \nonumber
  \cE_{\Ppm}(g)
  \ge \cE_{\Ptilpm}(g)
  &=
\frac{1}{2}\int \pi(\theta) w\qtil(w) q(\theta'|\theta)\qtil(w')
\alphatil_{pm}(\theta,w;\theta',w') \{g(\theta',w')-g(\theta,w)\}^2\md \theta \md \theta' \md w \md w'\\
&=\frac{1}{2}\int \qtil(w)\qtil(w') (w\wedge w')D(w,w') \md w \md w',
\label{eqn.defineEPtil}
\end{align}
where
\[
D(w,w'):=\int \pi(\theta) q(\theta'|\theta)\alpha_{MH}(\theta,\theta') \{g_{w'}(\theta')-g_w(\theta)\}^2 \md \theta \md \theta'.
\]
We would like to use the Dirichlet form for $\PMH$ to simplify $D$, but we cannot since the function acting on $\theta'$ is not the same as the function acting on $\theta$. The issue is resolved by the following, whose proof is purely algebraic manipulation and is given in Appendix \ref{sec.proof.lem.rearrange.square}:
\begin{lemma}
  \label{lem.rearrange.square}
  For any two functions $f,g:\cX\to \mathbb{R}$ and a third function $z:\cX\times \cX \to [0,\infty)$ with $z(x,y)=z(y,x)$,
  \[
\iint_{\cX^2} z(x,y)\{f(x)-g(y)\}^2\md x \md y \ge \frac{1}{4} \iint_{\cX^2}z(x,y)[\{f(x)+g(x)\}-\{f(y)+g(y)\}]^2\md x \md y. 
  \]
\end{lemma}
Letting $z(\theta,\theta')=\pi(\theta)q(\theta'|\theta)\alpha_{MH}(\theta,\theta')=z(\theta',\theta)$, applying Lemma \ref{lem.rearrange.square} and writing $(g_w+g_w')$ for the function $\theta:\to g_w(\theta)+g_{w'}(\theta)$, we obtain
\begin{align*}
D(w,w')
&\ge
\frac{1}{4}\int \pi(\theta)q(\theta'|\theta) \alpha_{MH}(\theta,\theta')\{(g_w+g_{w'})(\theta')-(g_w+g_{w'})(\theta)\}^2\\
&=
\frac{1}{2}\cE_{\PMH}(g_w+g_{w'})\\
&\ge
\frac{1}{2} \epsMH \|g_w+g_{w'}\|^2_{L^2(\pi)}.
\end{align*}
Thus
\begin{equation}
  \label{eqn.DirichletBoundMain}
  \cE_{\Ppm}\ge
  \frac{1}{4}\epsMH\int \qtil(w)\qtil(w') \{w\wedge w'\} \|g_w+g_{w'}\|^2_{L^2(\pi)}\md w \md w'.
\end{equation}

\textbf{Stage 2}: Using the fact that $\int_\cW w' \qtil(w') \md w'=1$, the inner product is
\begin{align*}
  \langle h_*, g\rangle_{L^2(\pi\times \nu)}
  &=
  \int \pi(\theta)w\qtil(w)  h_*(\theta) g_w(\theta) \md \theta \md w\\
  &=
  \int w\qtil(w) \langle h_*, g_w\rangle_{L^2(\pi)} \md w\\
  &=
  \int w\qtil(w) w'\qtil(w')\langle h_*, g_w\rangle_{L^2(\pi)} \md w\md w'\\
  &=
  \frac{1}{2}
 \int w\qtil(w) w'\qtil(w')\langle h_*, g_w+g_{w'}\rangle_{L^2(\pi)} \md w\md w', 
\end{align*}
by relabelling $w\leftrightarrow w'$ and averaging.

The Cauchy-Schwarz inequality then provides the bound
\begin{equation}
  \label{eqn.innerProdBoundMain}
  \langle h_*, g\rangle_{L^2(\pi\times \nu)}
  \le
  \frac{1}{2}
  \int \qtil(w)\qtil(w') w w' \|h_*\|_{L^2(\pi)} \|g_w+g_{w'}\|_{L^2(\pi)}.
\end{equation}

\textbf{Stage 3}: Combining \eqref{eqn.DirichletBoundMain} and \eqref{eqn.innerProdBoundMain} through \eqref{eqn.varRepvar}, $\asVar_{\Ppm}(h_*) +\|h_*\|^2_{L^2(\pi)}$ is
\begin{align*}
&\le
\int\qtil(w)\qtil(w')
\left\{2 w w' \|h_*\|_{L^2(\pi)} \|g_w+g_{w'}\|_{L^2(\pi)}
-\frac{1}{2} \epsilon_{\mathrm{MH}}(w\wedge w')\|g_w+g_{w'}\|^2_{L^2(\pi)}\right\}\md w \md w'\\
&\le
\int\qtil(w)\qtil(w') \frac{4\|h_*\|^2_{L^2(\pi)}w^2{w'}^2}{2\epsilon_{\mathrm{MH}}(w\wedge w')} \md w \md w'\\
&=
\frac{2}{\epsilon_{\mathrm{MH}}}\|h_*\|^2_{L^2(\pi)}\int_{\cW^2}\qtil(w)\qtil(w') ww'(w \vee w') \md w \md w',
\end{align*}
since $w w' = (w\vee w')(w\wedge w')$. 
The second inequality follows as when $c> 0$, $ax-cx^2=-c[x-a/(2c)]^2+a^2/(4c) \le a^2/4c$. $\square$

\subsection{Comparison with other bounds in the literature}
We now consider the case where Assumptions \ref{ass.indepNoise} and \ref{ass.lognormalCLT} hold and derive a closed-form expression for the bound in Theorem \ref{theorem.varboundA} in this case. We then compare this bound with two other, equally general bounds, taken from \cite{ALPW2022} and \cite{DPDK2015}. 

\begin{prop}
  \label{prop.PMCLT}
  If Assumptions \ref{ass.indepNoise} and \ref{ass.lognormalCLT} hold. Then \eqref{eqn.aVarUB} becomes
  \[
R_{\mathrm{S}}(\sigma):=\iint_{\cW^2} \qtil(w)\qtil(w') w w' \{w\vee w'\} \md w \md w'=2\exp(\sigma^2)\Phi\left(\frac{\sigma}{\sqrt{2}}\right).
\]
\end{prop}

Example 60 of \cite{ALPW2022} provides an upper bound on $\asVar_{\Ppm}$ derived from a weak Poincar\'e inequality. Translated into our notation, this becomes $\asVar_{\Ppm}+\|h_*\|^2_{L^2(\pi)}\le \|h_*\|^2_{\mathrm{osc}}\times  \frac{4}{C_p}R_{\mathrm{ALPW22}}(\sigma)$, where $C_P$ is the right-spectral gap of $\PMH^2$, and 
\[
R_{\mathrm{ALPW22}}(\sigma):=
2\sqrt{2 \pi}\exp(\sigma^2)\frac{1+\sigma^2}{\sigma} \Phi(\sigma)+2\exp(\sigma^2/2).
\]
Its form bears a striking resemblance to $R_{\mathrm{S}}(\sigma)$, and can be made directly comparable for kernels where the left spectral gap is at least as large as the right spectral gap and the right spectral gap, $\epsilon_{MH}$, is small: in this case $C_P\approx 2\epsilon_{MH}$. Since $\|f\|^2_{\mathrm{osc}}\ge \|f\|^2_{L^2(\pi)}$, we may compare $R_{\mathrm{S}}$ against a best-case equivalent from the bound $R_{\mathrm{ALPW22}}$.

We also examine the general bound in Corollary 1 of \cite{DPDK2015}. Let 
\begin{align*}
\alpha_1(w,w')&=1\wedge w'/w,~~~
\alphabar_1(w)=\int_{\cW}\qtil(w')\alpha_1(w,w')\md w'~~~
\mbox{and}~~~
\widetilde{\alpha}_1^{-1}=\int_{\cW}\qtil(w) w \frac{1}{\alphabar_1(w)} \md w.
\end{align*}
After some manipulation, the general bound in Corollary 1 of \cite{DPDK2015} can be simplified considerably to provide a ratio between $\|h_*\|^2_{L^2(\pi)}+\asVar_{\Ppm}(h_*)$ and $\|h_*\|^2_{L^2(\pi)}+\asVar_{\PMH}(h_*)$ of 
\[
R_{\mathrm{DPDK15}}(\sigma)=\frac{2}{\widetilde{\alpha}_1}-\int_{\cW}w \qtil(w)\qtil(w')\frac{\alpha_1(w,w')}{\alphabar_1(w)\alphabar_1(w')}\md w \md w'. 
\]
Subject to Assumption \ref{ass.lognormalCLT},
\[
\alphabar_1(w)=\frac{1}{w}\Phi\left(-\frac{1}{2}\sigma+\frac{1}{\sigma}\log w \right)+\Phi\left(-\frac{1}{2}\sigma - \frac{1}{\sigma}\log w\right).
\]
The harmonic mean  $\widetilde{\alpha}_1$ then requires only a single numerical integral for each value of $\sigma$, and the second quantity requires a double numerical integral for each $\sigma$.

Unlike $R_{\mathrm{S}}$ and $R_{\mathrm{ALPW22}}$, which bound the ratio of the worst case under $\Ppm$ to the worst case under $\PMH$, $R_{\mathrm{DPDK15}}$ bounds the ratio for each individual function. However, since $h_*$ does not depend on $w$, we imagine $R_{\mathrm{S}}$ and $R_{\mathrm{ALPW22}}$ might be representative bounds of the ratio for individual functions, and it is in this spirit that we compare the three bounds in the right panel of Figure \ref{fig.compareBounds}. The bound from \cite{ALPW2022} is always at least a factor of $6$ larger than $R_{\mathrm{S}}(\sigma)$, but $R_{\mathrm{S}}(\sigma)$ and $R_{\mathrm{DPDK15}}(\sigma)$ are almost indistinguishable. The left panel depicts the ratio $R_{\mathrm{DPDK15}}(\sigma)/R_{\mathrm{S}}(\sigma)$ and shows that, at least over the range of $\sigma$ values of interest, neither dominates the other and they are within $2\%$ of each other. $R_{\mathrm{S}}(\sigma)$, however, requires no numerical integration and was more straightforward to obtain.

\begin{figure}
\begin{center}
  \includegraphics[scale=0.4,angle=0]{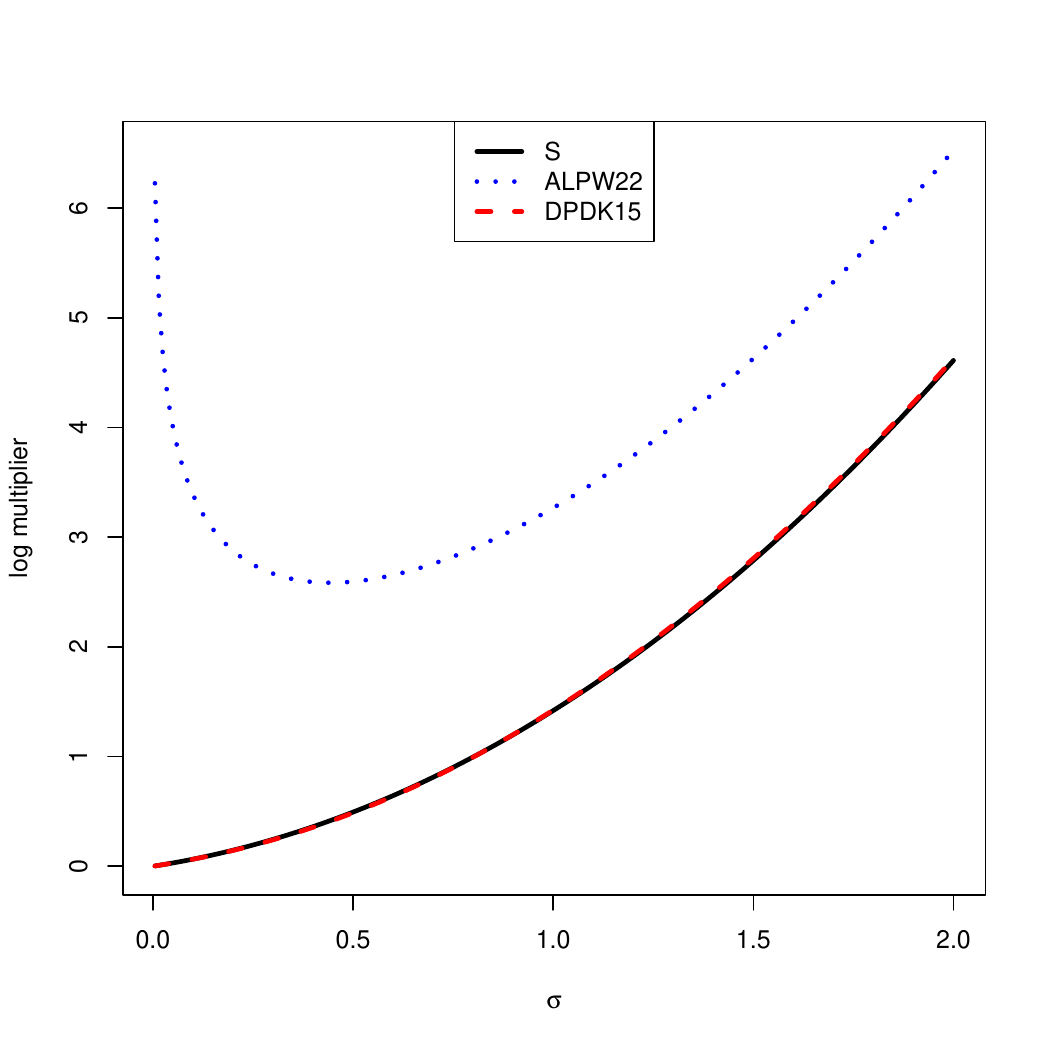}
  \includegraphics[scale=0.4,angle=0]{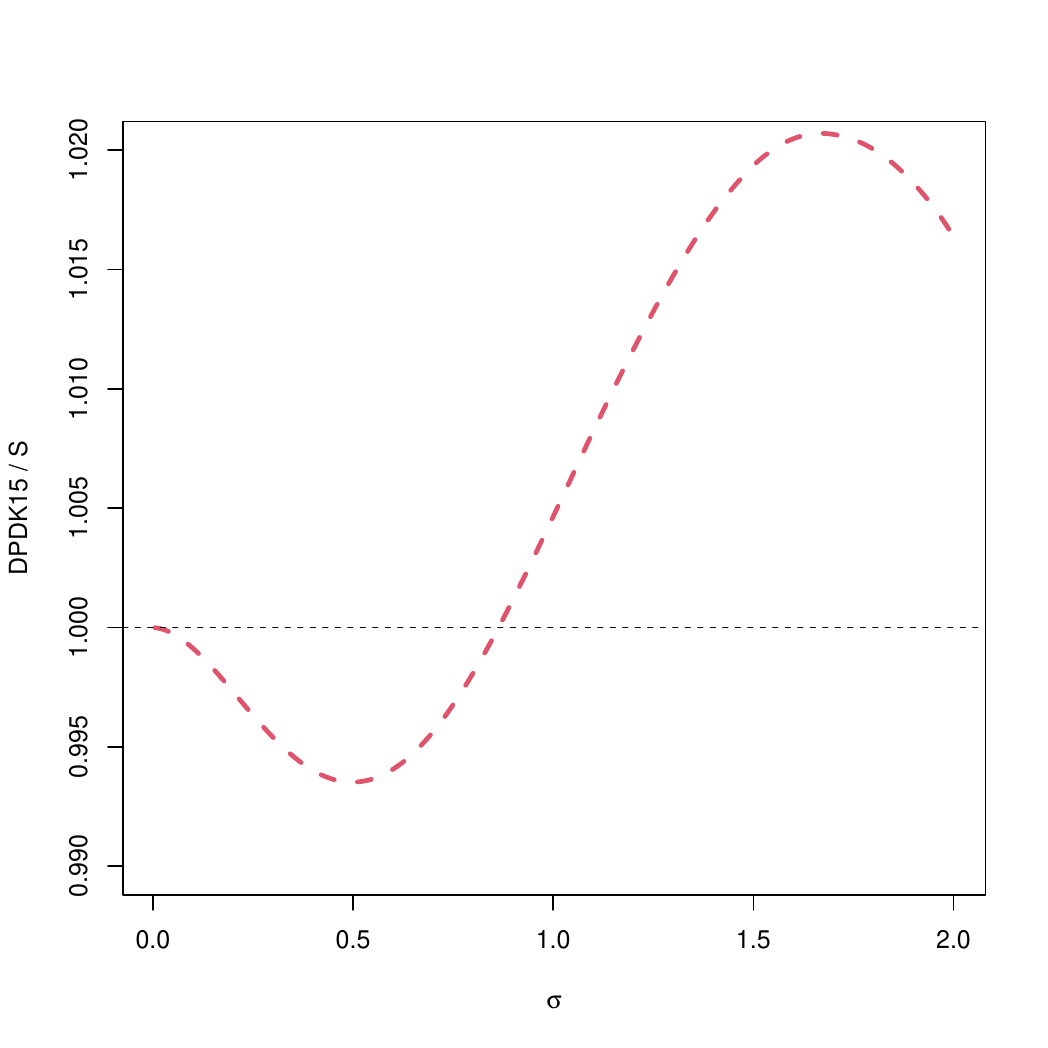}
  \caption{Left: logarithm of the bound on the multiplier of $\|h_*\|^2+\asVar_{\PMH}(h^*)$ against $\sigma$ for $R_S(\sigma)$ from Proposition \ref{prop.PMCLT} and for the bounds in \cite{ALPW2022} (ALPW22) and \cite{DPDK2015} (DPDK15). Right: ratio of the multiplier bound from \cite{DPDK2015} to that in Proposition \ref{prop.PMCLT}.
\label{fig.compareBounds}
}
\end{center}
\end{figure}

The quantity $\widetilde{\alpha}_1^{-1}$ can be rewritten as
\[
\widetilde{\alpha}_1^{-1}=\int_{\cW} \qtil(w) w^2 \frac{1}{\int_{\cW}\qtil(w') (w\wedge w') \md w'} \md w.
\]
Now that we know what to look for, we can see that $\widetilde{\alpha}_1^{-1}=\infty$ if $\Expects{\qtil}{W^2}=\infty$ since $\lim_{w\to \infty}\int_{\cW} \qtil(w')(w\wedge w')\md w'=1$.

\cite{DPDK2015} contains a lower bound for a general kernel: $\asVar_{\Ppm}(h) \ge \asVar_{\PMH}(h)/\Expects{\nu}{\alphabar_1(W)}$. Unfortunately, this is finite even when $\Expects{\qtil}{W^2}=\infty$. The article also provides tighter upper and lower bounds in the special case that $\PMH$ is positive, which we do not explore here since our interest is in bounds for general Metropolis--Hastings kernels.

\subsection{General $h$}

In the case that interest is in a general $h(\theta,w)\in L^2(\pi\times \nu)$, we present the following explicit upper bound, proved in Appendix \ref{app.proveBoundB}:

\begin{theorem}
\label{theorem.varboundB}
Suppose that $\PMH$ has a right-spectral gap of $\epsMH$ and Assumption \ref{ass.indepNoise} holds. For any $h\in L^2(\pi\times \nu)$, write $h_w(\theta)\equiv h(\theta,w)$ as the function of $\theta$, indexed by $w$. Then $h_w(\theta)\in L^2(\pi)$ $\nu$-almost surely, and 
  \begin{align*}
  \asVar_h(\Ppm) &\le
  \frac{2}{\epsMH}\iint_{\cW\times \cW} \qtil(w)\qtil(w')ww' \{w\vee w'\}\left\{\|h_{w}^2\|_{L^2(\pi)}+\|h_{w'}\|^2_{L^2(\pi)}\right\}\md w \md w' -\|h\|^2_{L^2(\pi\times \nu)}\\
  &=
  \frac{4}{\epsMH}\iint_{\cW\times \cW} \qtil(w)\qtil(w')ww' \{w\vee w'\}\|h_{w}^2\|_{L^2(\pi)}\md w \md w' -\|h\|^2_{L^2(\pi\times \nu)}.  
 \end{align*}
 \end{theorem}

Proposition 4 of \cite{DelLee2018} provides the following sufficient condition for $\asVar_{\Ppm}(h)<\infty$:
\[
\int_{\cW}(w+\Expects{\qtil}{W^2})w \left\{\int_{\cX} \frac{h(\theta,w)^2}{\alphabar_{\mathrm{MH}}(\theta)}\pi(\theta) \md \theta\right\}\qtil(w)\md w <\infty,
\]
where $\alphabar_{\mathrm{MH}}(\theta)=\int_{\cX}q(\theta'|\theta)\alpha_{\mathrm{MH}}(\theta,\theta')\md \theta'$. If $\PMH$ is geometrically ergodic then \cite[e.g.][Theorem 5.1]{RobTwe1996} $\alphabar_{\mathrm{MH}}>\delta>0$ and this condition becomes equivalent to the sufficient condition implied by the explicit bound in Theorem \ref{theorem.varboundB} because $(w+w')/2\le w\vee w' \le w+w'$.

\subsection{Correlated pseudo-marginal Metropolis--Hastings}
\label{sec.corrPMMH}
The correlated pseudo-marginal method \cite[]{DDP2018,DLKS2015} makes $W'$ positively correlated with $W$ rather than independent of it. This reduces the variance of the ratio $W'/W$, which, in turn, increases the acceptance rate if the number of particles is kept fixed or allows a reduced number of particles whilst maintaining the same acceptance rate. The noise $W'$ is proposed from a density $\qtil(w'|w)$, such that
\[
\qtil(w,w'):=\qtil(w)\qtil(w'|w)
=
\qtil(w',w);
\]
that is, if $W\sim \Qtil$ then $W$ and $W'$ come from an exchangeable distribution.

To control the variance of the PMMH acceptance probability when likelihood estimates are obtained through a particle filter, the number of particles must increase in proportion to the number of observation times. \cite{DDP2018} shows that for correlated PMMH it is sufficient that the number of particles increases sublinearly. The asymptotic variance of an approximate kernel motivated through Assumption \ref{ass.lognormalCLT} and an averaging behaviour, and which might be expected to `capture come of the quantitative properties of correlated pseudo-marginal kernel' is then upper bounded through a handicapped version of the kernel, leading to tuning advice. We show that converting from PMMH to the correlated pseudo-marginal method can fundamentally change the nature of the convergence so that infinite asymptotic variances become finite.

Given the positive correlation between $W$ and $W'$, it is not unreasonable to assume that any particular moment of $W'$ increases with $W$, motivating the following assumption:

\begin{assumption}
  \label{eqn.assump.corrPM}
There is a function $b:\cW\to \mathbb{R}^+$ and $c_{b}>0$ such that
\begin{equation}
  \label{eqn.corrPMcond}
\Expects{\qtil}{{W'}^{1/2}b(W')|W=w}\ge c_{b} \frac{w^{1/2}}{b(w)}~~~\mbox{for}~\nu~\mbox{almost all}~w\in \cW.
\end{equation}
\end{assumption}

The pseudo-marginal algorithm satisfies Assumption \ref{eqn.assump.corrPM} with $b(w)=w^{1/2}$ and $c_b=1$. For a correlated PMMH algorithm with a positive correlation, we expect \eqref{eqn.corrPMcond} to hold for some $b(w)$ that increases more slowly than $w^{1/2}$. Indeed if $(\log W,\log W')$ satisfy a joint central limit theorem then $b(w)$ relates directly to their correlation (see Appendix \ref{sec.prove.corrPMCLT} for the proof):

\begin{prop}
  \label{prop.corrPM.CLT}
If under $q(w,w')$,
  \[
  \begin{bmatrix}
    \log W\\
    \log W'
  \end{bmatrix}
  \sim
  \mathsf{N}\left(
  \begin{bmatrix}
    -\frac{1}{2}\sigma^2\\
    -\frac{1}{2}\sigma^2
  \end{bmatrix}
  ,
  \begin{bmatrix}
    \sigma^2,\rho\sigma^2\\
    \rho\sigma^2,\sigma^2
  \end{bmatrix}  
  \right).
  \]
then Assumption \ref{eqn.assump.corrPM} holds with $b(w)=w^{(1-\rho)/(2+2\rho)}$ and $c_b=1$.  
\end{prop}

Our main result for correlated pseudo-marginal Metropolis--Hastings is:
\begin{theorem}
\label{theorem.varboundCorrPM}
If Assumption \ref{eqn.assump.corrPM} holds for the correlated pseudo-marginal kernel $\Pcpm$ and if $\PMH$ has a right-spectral gap of $\epsMH$ then for any $h_*\in L^2(\pi)$, 
\[
\asVar_{\Pcpm}(h_*) \le
\left[  \frac{2 c_\gamma^{-2}}{\epsMH}\iint_{\cW\times \cW} q(w,w') \{w\vee w'\}b(w)^2b(w')^2\md w \md w' -1\right]
\|h_*\|_{L^2(\pi)}.
\]
\end{theorem}

As pointed out in Section \ref{sec.TheBounds}, $w\vee w'$ is bounded between $(w+w')/2$ and $w+w'$; so an equivalent sufficient condition for finite asymptotic variance is $\Expects{\qtil}{Wb(W)^2}<\infty$. In the case of Proposition \ref{prop.corrPM.CLT}, this amounts to requiring $\Expects{\qtil}{W^{2/(1+\rho)}}<\infty$; \emph{i.e.}, if $\rho>0$  a finite second moment for $W$ is no longer necessary for a finite asymptotic variance.

\section{Tuning particle Metropolis--Hastings}

As mentioned Section \ref{sec.intro}, the particle Metropolis--Hastings algorithm obtains the realisation of a non-negative unbiased estimator of the likelihood from a particle filter. Assumptions \ref{ass.indepNoise} and \ref{ass.lognormalCLT} are justified for large $T$ and large numbers of particles by \cite{SDDP2021} and \cite{BDD2014}, respectively, and in this case both both \cite{STRR2015} and \cite{DPDK2015} recommend choosing a number of particles so that $\Var{\log \pihat(\theta)} \approx 1$ for each $\theta$. Since, in practice this variance does depend on $\theta$, both articles recommend picking a value $\thetahat$ that is representative of the main posterior mass and choosing $n$ so that $\Var{\log \pihat(\thetahat)}\approx 1$. In fact, the optimisations suggest choosing $\Var{\pihat(\theta)}$ between $0.9$ and $3.3$; however, again because $\Var{\log \pihat(\theta)}$ is often larger in the tails of the posterior and because the mixing penalty associated with a larger variance can be severe, the  practical recommendation is to err on the side of caution.

Under Assumptions \ref{ass.indepNoise} and \ref{ass.lognormalCLT}, the efficiency measure from Theorem \ref{theorem.varboundA} is proportional to $\sigma^{-2}\{2R_{\mathrm{S}}(\sigma)-\epsilon_{\mathrm{MH}}\}$. Figure \ref{fig.lognormal.CLT} plots this function for various $\epsilon_{\mathrm{MH}}$. Empirically, we find this efficiency measure is minimised at $\sigma_{opt}\approx (0.83,0.88,0.91,0.92,0.93)$ respectively when $\epsilon_{\mathrm{MH}}= (1,0.5,0.2,0.05,0)$. Given the closeness of $R_{\mathrm{DPDK15}}$ to $R_{\mathrm{S}}$, the corresponding efficiency measure from \cite{DPDK2015} is minimised at almost identical values.

\begin{figure}
\begin{center}
  \includegraphics[scale=0.6,angle=0]{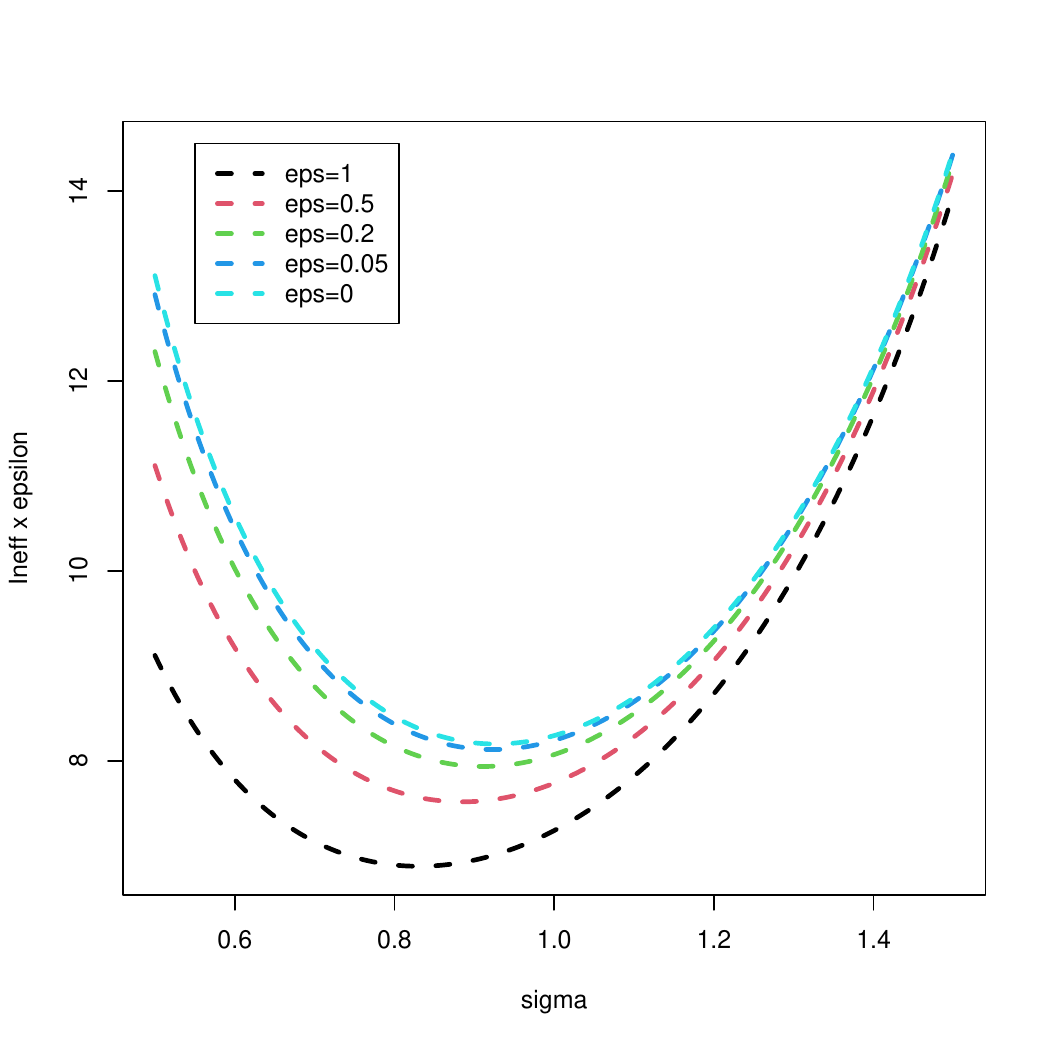}
  \caption{Relative inefficiency against $\sigma$ for the lognormal central limit theorem regime plotted for various values of $\epsilon_{\mathrm{MH}}$, the right spectral gap of the Metropolis-Hastings chain.
\label{fig.lognormal.CLT}
}
\end{center}
\end{figure}

Now, $\Var{\log \pihat(\theta;U)}\equiv \Var{\log W}$. The results and discussion in Section \ref{sec.VarBounds}, however, all point to $\Expect{W^2}$ being the critical quantity. In particular, $\Var{\log W}$ can be finite and estimators of it based on a sample can be well behaved, even when $\Expect{W^2}=\infty$. For example, consider the density in \eqref{eqn.heavy}, which we will return to for a numerical illustration in Section \ref{sec.num.heavy}. If $a\le 2$ then the second moment is infinite; however all polynomial moments of $\log W$ are finite. As part of the tuning procedure, we would try to estimate $\Var{\log W}$, and we would find a consistent estimator for this. There would be no warning that the asymptotic variance of any function $h_*\in L^2(\pi)$ was infinite. 

Since $\Expect{W}=1$, from the delta method, $\Vars{\qtil}{\log W}\approx \Vars{\qtil}{W}$, provided the mass for $W$ is concentrated around its expectation; \emph{i.e.}, provided $W$ is well behaved. This suggests tuning the particle filter according to $\Vars{\qtil}{W}$. If this can be estimated consistently then it suggests that the particle filter is sufficiently well behaved that asymptotic variances of quantities of interest will be finite. On the other hand, if it is difficult to esimate $\Vars{\qtil}{W}$ consistently, this suggests not running the PMMH algorithm at all and, instead, trying to find a better-behaved particle filter. For example, motivated by the results in Section \ref{sec.corrPMMH}, it might be possible to implement a correlated version of the PMMH algorithm; alternatively a more efficient proposal scheme might for the transitions could be employed \cite[e.g.][]{WGBS2017}.

Under Assumption \ref{ass.lognormalCLT}, when $\Var{\log W}=\sigma^2=0.9$, $\Var{W}=\exp(\sigma^2)-1\approx 1.5$. We, therefore, suggest the following:
\vspace{-.3cm}
\begin{itemize}
\item Choose the number of particles so that $\Var{W}\approx 1.5$.
\end{itemize}
\vspace{-.3cm}
In practice, we estimate $\Var{W}=\Var{\pihat(\theta;U)/\pi(\theta)}$ by obtaining $M$ estimates $\pihat(\thetahat;u^1)$, $\dots$, $\pihat(\thetahat;u^M)$ of $\pi(\thetahat)$ through repeated runs of the particle filter and then evaluating the quotient of the sample variance and the square of the sample mean. 

The moments of $W$ are also strongly linked with the polynomial convergence rate of the PMMH algorithm. For example Corollary 45 of \cite{ALPW2022} shows that if $\PMH$ is geometrically ergodic and $\Expects{\pi}{\Expects{\qtil_{\theta}}{W^k}}<\infty$ then $\Ppm$ converges to equilibrium at a rate of at least $n^{-k}$. Hence, if it is straightforward to obtain a reliable estimate of $\Var{W}$, then we might expect a better rate of convergence.

Tuning according to $\Var{W}$ also bypasses the logical contradiction exemplified in the example and generalisations of Section \ref{sec.exactObs}, where $\Var{\log W}$ is not even defined.

\section{Numerical illustrations}
\subsection{Checks on the bound}
\label{sec.simple.tests}
We now verify the bound from Theorem \ref{theorem.varboundA} empirically for examples specially constructed so that the true Metropolis--Hastings algorithm is implementable.  In both cases we measure efficiency in terms of the expected sample size, $\ESS:=n_{its}/\asVar_{\mathsf{P}}(h)$ with $h(\theta)=\theta$. In the first example both Assumptions \ref{ass.indepNoise} and \ref{ass.lognormalCLT} hold, whereas in the second example, neither assumption holds. In both examples 
we set $\theta=(\theta_1,\theta_2,\theta_3)$, $\pi(\theta)\propto \exp(-\|\theta\|^2/2)$ and use a proposal $\theta'\sim \mathsf{N}(\theta,\lambda^2 I_3/3)$ with $\lambda=1.4$ and an initial value of $\theta=(0,0,0)^\top$. Even in this simple example, $\cE_{\PMH}\approx 0.1$.

We start with a test of the case when the lognormal CLT holds precisely:
\[
\pihat(\theta)\propto \pi(\theta)\exp\left(-\frac{1}{2}\sigma^2 +\sigma Z\right)
\]
for various values of $\sigma$ with $Z\sim \mathsf{N}(0,1)$ independent across draws. We perform $\PMH$ on $\pi$ and then, for each $\sigma$ we perform both $\Ppm$ and $\Ptilpm$ on $(\pihat,Z)$; all algorithms are run for $n_{its}=2\times 10^5$ iterations.

When, as is usually the case, $\epsMH<<1$, $R_S^{\qtil}$ gives an approximate bound on the ratio of the worst case under $\PMH$ to the worst case under $\Ppm$. If this ratio can be applied to individual functions then for any particular $h$, $\ESS_{\mathrm{pm}}(h)\le \ESS_{\mathrm{MH}}(h)/R(\sigma)$. Of course, the bound that we have is actually for $\Ptilpm$, so we compare to that, too. The left panel of Figure \ref{fig.toyCLTandBinomial} plots the ESSs for both $\Ppm$ and $\Ptilpm$ against $\sigma$. It also shows the quotient of the ESS for $\PMH$ and $R_S(\sigma)$. The best that can be hoped for is for the curves based on the theory to mimic the performance of the handicapped kernel $\Ptilpm$; this is achieved for small values of $\sigma$, but the theory underestimates the ESS (hence, over-estimates the asymptotic variance) for larger values of $\sigma$.

\begin{figure}
\begin{center}
  \includegraphics[scale=0.42,angle=0]{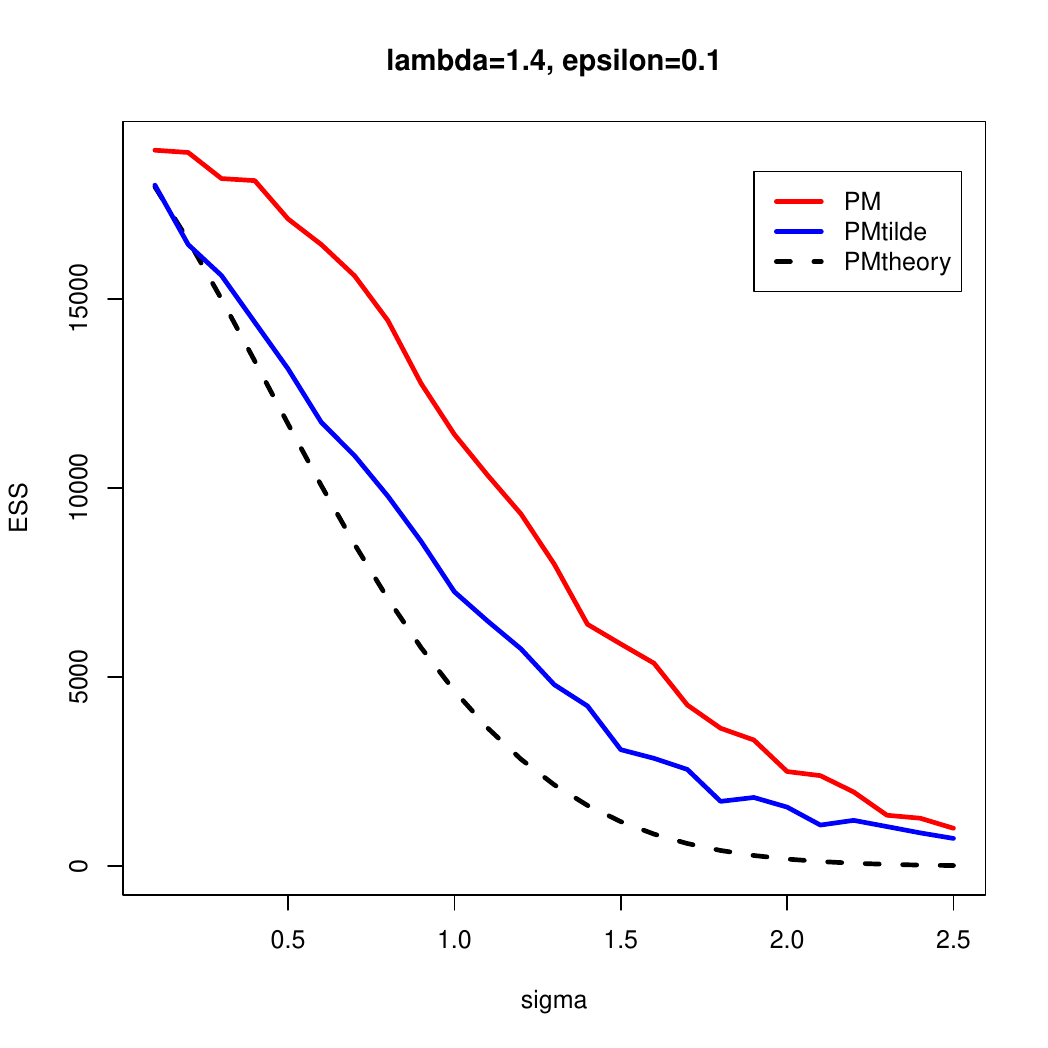}
  \includegraphics[scale=0.42,angle=0]{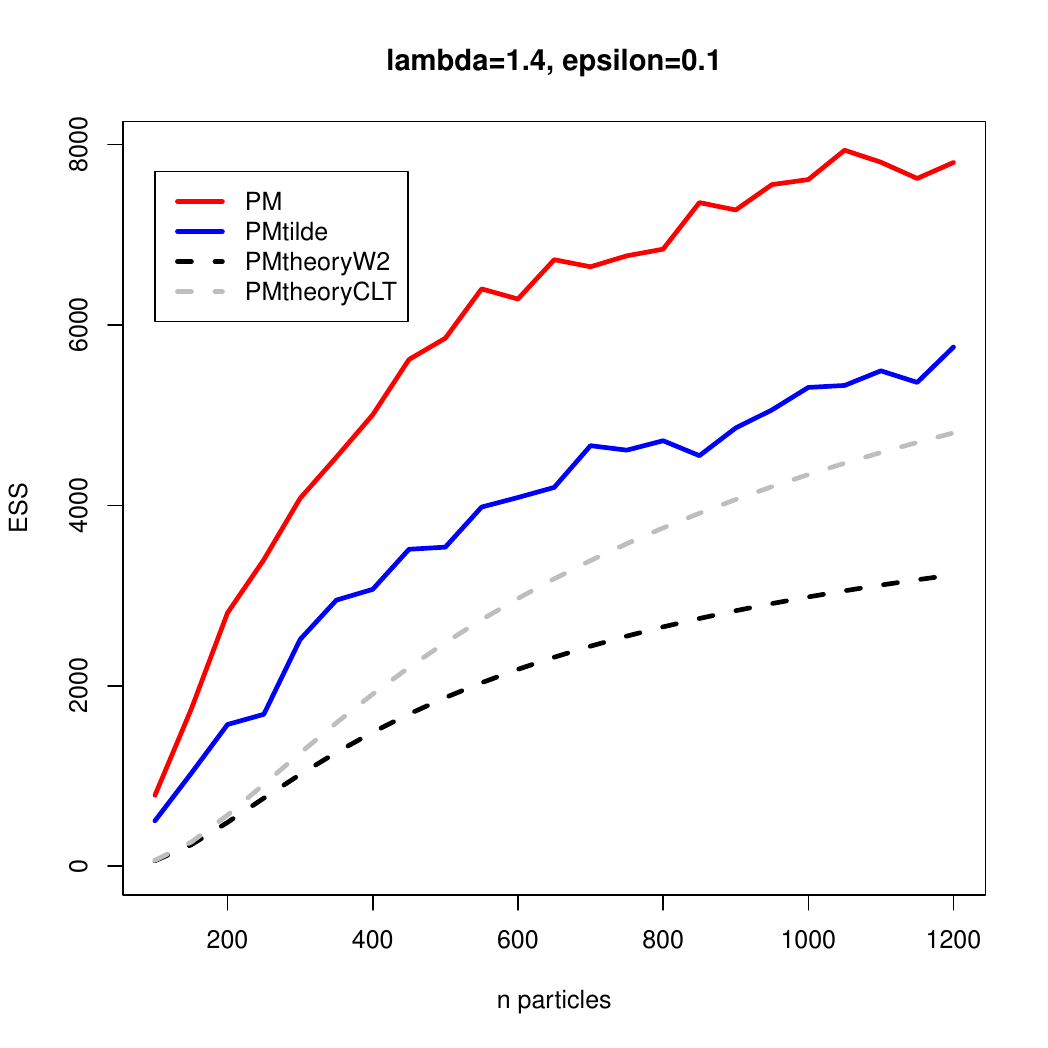}
  \caption{Left panel: Effective sample size (ESS) against $\sigma$, the standard deviation of the Gaussian additive noise in $\log W$, for $\Ppm$ and $\Ptilpm$ together with with $\ESS_{\mathrm{MH}}/\{2R_S(\sigma)\}$. Right panel: product of binomial estimators; ESS, both for $\Ppm$ and $\Ptilpm$, and $\ESS_{\mathrm{MH}}/\{2\Expects{\qtil}{W_n^2}\}$ and $\ESS_{\mathrm{MH}}/R(\widetilde{\sigma}_n)$, where $\exp(\widetilde{\sigma}_n)=\Expect{W_n^2}$, all against number of particles.
\label{fig.toyCLTandBinomial}
}
\end{center}
\end{figure}

Next, imagine $T=30$ exact observations and a known initial condition, $x_0$. We set
\[
\Prob{X_{t}=x_t|X_{t-1}=x_{t-1}}=p_t\exp\left\{-\frac{\|\theta\|^2}{2T}\right\},
\]
where $p_t\stackrel{iid}{\sim}\mathsf{Beta}(5,45)$, $t=1,\dots,T$. We place a uniform prior on $\theta$, so
\[
\pi(\theta)\propto \Prob{X_{1:T}=x_{1:T}|X_0=x_0}=\prod_{t=1}^T \left[p_t \exp\left(-\frac{1}{2T}\|\theta\|^2\right)\right]
\propto
\exp\left(-\frac{1}{2}\|\theta\|^2\right).
\]
We create an unbiased estimator for $\Prob{X_{1:T}=x_{1:T}|X_0=x_0}$ as $\prod_{t=1}^T \Phat_t$, where
\[
\Phat_t \sim \mathsf{Bin}\left(n,p_t\exp\left\{-\frac{1}{2T}\|\theta\|^2\right\}\right).
\]
In this case $R_S^{\qtil}$ in \eqref{eqn.aVarUB} is intractable, but the looser bound of $R_S^{\qtil}<2\Expects{\qtil}{W^2}$ is tractable. With $n$ particles it is
\[
2\Expects{\qtil}{W_n^2}=2\prod_{t=1}^T\left\{1+\frac{1-p_t}{np_t}\right\}.
\]
If both Assumptions \ref{ass.indepNoise} and \ref{ass.lognormalCLT} held then $\Expect{W_n^2}=\exp(\sigma^2)$. In fact, neither of these assumptions holds; nonetheless, equating $\Expect{W_n^2}$ and $\exp(\sigma^2)$ gives a nominal value $\widetilde{\sigma}_n$ and hence an approximation  $R_S^{\qtil} \approx R_S(\widetilde{\sigma}_n)$. 

We use $\PMH$ and then $\Ppm$ and $\Ptilpm$ for various values of $n$. All algorithms are run for $10^5$ iterations. For the latter kernels, the right panel of Figure \ref{fig.toyCLTandBinomial} plots $\ESS$ against $n$. We also plot both $\PMH/(2\Expect{W^2})$ and $\PMH/R_S(\widetilde{\sigma}_n)$ against $n$. As for the CLT-based example, the theory mimics the true performance well for large $n$ (smaller noise variance) but underestimates the performance more substantially for small $n$ (larger $\sigma$).

\subsection{Heavy tails}
\label{sec.num.heavy}
We provide a short numerical illustration of the heavy-tail issue and how using correlated PMMH can solve it. We use the three-dimensional Gaussian posterior in Section  \ref{sec.simple.tests}, $\pi(\theta)\propto \exp(-\|\theta\|^2/2)$, with a random-walk Metropolis proposal $\theta'\sim \mathsf{N}(\theta,\lambda^2 I_d)$ and $\lambda=1.4$. However we use the proposal distribution for $W$ with the density \eqref{eqn.heavy}. This has a finite second moment precisely when $a>2$. For three replicates in each of three scenarios we ran the algorithm for $10^6$ iterations $M=50$ times. At each iteration number, $j$, we obtained the variance of the $M$ estimates $\frac{1}{j}\sum_{i=1}^j h(\theta_i)$, where $h(\theta)$ extracts the first component of the vector $\theta$.

The first two scenarios have $a=1.5$ and $a=2.5$ respectively. The third scenario has $a=1.5$ but uses the following correlated pseudo-marginal algorithm: given $W$, set
\begin{equation}
  \label{eqn.heavyCorr}
  E=2a\log(1+W),
  ~~E'=E\cos^2 U +Z^2
  ~~\mbox{and}~~W'=\exp\{E/(2a)\}-1,
\end{equation}
where $U\sim \mathsf{Unif}[0,2\pi)$, $Z\sim \mathsf{N}(0,1)$ and $W$ are independent. Since the cumulative distribution function of $W$ under $\qtil$ is $F_{\qtil}(w)=1/(1+w)^a$, $E\sim \mathsf{Exp}(1/2)$. That $E'\sim \mathsf{Exp}(1/2)$ follows from the Box-Muller transformation \citep[]{BoxMuller1958}. This is stated formally in Proposition \ref{prop.corrPMheavy}, below, together with the manner in which it satisfies Assumption \ref{eqn.assump.corrPM}. 

  \begin{prop}
    \label{prop.corrPMheavy}
    The algorithm given in \eqref{eqn.heavyCorr} has the stationary distribution with the density given in \eqref{eqn.heavy} and satisfies Assumption \ref{eqn.assump.corrPM} with $b(w)=w^{1/6}$.
  \end{prop}
  
  This second part of the proposition is proved in Appendix \ref{sec.proveCorrHeavy}. From Theorem \ref{theorem.varboundCorrPM}, therefore, only $\Expects{\qtil}{W^{4/3}}$ needs to be finite to imply a finite asymptotic variance. In particular, then, $a=1.5$ is sufficient for a finite asymptotic variance for $h_*(\theta)$ when the correlated pseudo-marginal algorithm is used. Figure \ref{fig.HeavyTail} depicts the logarithm of the estimated asymptotic variances when $h_*(\theta)$ is the first component of $\theta$, as a function of the iteration number. It shows the lack of consistency for PMMH when $a=1.5$. It also suggests that, as well as being finite, the asymptotic variance under the correlated PMMH algorithm when $a=1.5$  might be lower than that under PMMH when $a=2.5$.

\begin{figure}
\begin{center}
  \includegraphics[scale=0.6,angle=0]{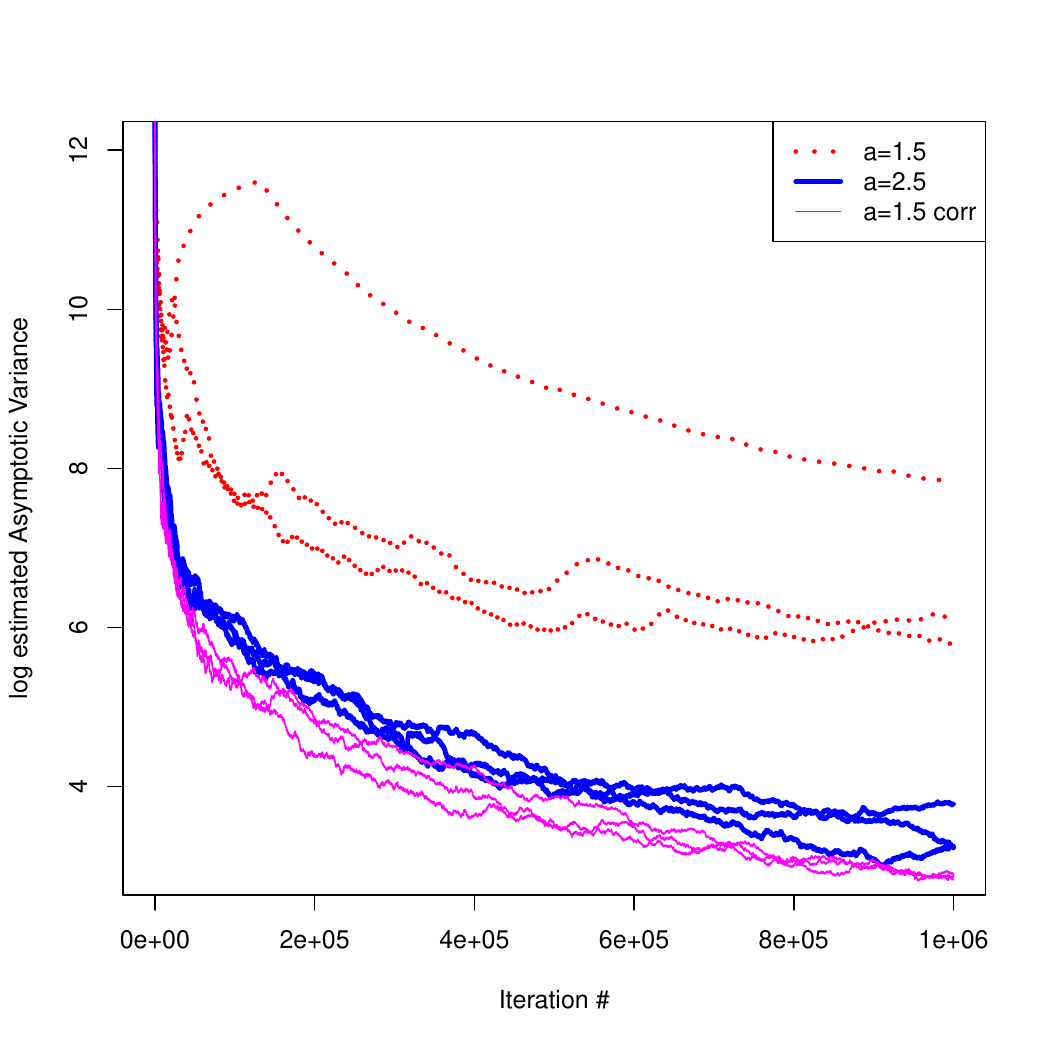}
  \caption{The logarithm of the estimated asymptotic variance as a function of iteration number from $50$ runs for each of three replicates in each of the scenarios: PMMH with $a=1.5$, PMMH with $a=2.5$ and correlated PMMH with $a=1.5$.
\label{fig.HeavyTail}
}
\end{center}
\end{figure}

\section{Discussion}
We have provided simple, explicit upper and lower bounds on the asymptotic variance of ergodic averages of a pseudo-marginal Metropolis--Hastings Markov chain. The bounds make it clear that finite asymptotic variances are closely linked with a finite variance of the multiplicative noise. When this variance is infinite, existing tuning advice can still be followed, but the resulting estimators of expectations will be poorly behaved. We suggest tuning according to the variance of the noise itself, rather than its logarithm. In well-behaved cases the two are essentially equivalent, but in badly behaved cases the new advice highlights issues that might otherwise have remained hidden. We have also shown that the asymptotic variance of ergodic averages under the correlated pseudo marginal algorithm can be finite even when those for the uncorrelated version are infinite, so that this technique could potentially rescue a poorly performing particle Metropolis--Hastings algorithm.

Theorem \ref{theorem.varboundA} and its extensions rely on Assumption \ref{ass.indepNoise}, that the proposal density, $\qtil_\theta(w)$, for the multiplicative noise does not depend on $\theta$. Whilst this appears to be a reasonable approximation in many asymptotic regimes, it would be preferable to cover the more general noise proposal mechanism. Unfortunately, Assumption \ref{ass.indepNoise} is key to obtaining the bound on the Dirichlet form of $\Ppm$ that is central to the proofs of Theorems \ref{theorem.varboundA}, \ref{theorem.varboundB} and \ref{theorem.varboundCorrPM}. We conjecture, therefore, that corresponding simple bounds are not obtainable under the more general noise regime. 

The bound in Theorem \ref{thrm.necessary} is based on failures to accept, so it might be expected to be tight in cases where the underlying Metropolis--Hastings jumps, when they occur, are large (compared with the length scales of the posterior) and the proposal is such that the acceptance rate is typically low. Its main purpose is to clearly illustrate the necessity of $\Expect{W^2}<\infty$ in many cases, so the fact that it is generally not tight is less important. A better bound would relate directly to the mixing of the underlying Metropolis--Hastings chain, perhaps using the jump chain, and is likely to be less transparent.

A finite $k$th moment of the multiplicative noise, $W$, directly implies polynomial convergence of (at least) order $k$ \cite[]{ALPW2022}. The fact that the correlated pseudo-marginal MCMC can require lower polynomial moments of $W$ for finite asymptotic variance suggests a different translation from the existence of a moment to the rate of convergence, suggesting a topic for future exploration.

\bibliographystyle{apalike}
\bibliography{PMMH.bib}

\appendix

\section{Additional results for Proof of Theorem \ref{theorem.varboundA}}

\subsection{Proof of Proposition \ref{prop.LsqInherit}}
\label{sec.proof.prop.LsqInherit}
\begin{proof}
 Suppose that there is a non-null set $A\in \cW$ such that $g_w(\theta)\notin L^2(\pi)$ for each $w\in A$. Then
    \[
    \|g(\theta,w)\|^2_{L^2(\pi\times \nu)}= \int_\cW \nu(\md w) \int_\cX \pi(\md \theta) g_w(\theta)^2
    \ge
    \int_A \nu(\md w) \|g_w\|^2_{L^2_{\pi}}= \infty,
    \]
    contradicting the fact that $g\in L^2(\pi\times \nu)$.
    \end{proof}

\subsection{Proof of Lemma \ref{lem.rearrange.square}}
\label{sec.proof.lem.rearrange.square}
We must show that for any two functions $f,g:\cX\to \mathbb{R}$ and a third function $z:\cX\times \cX \to [0,\infty)$ with $z(x,y)=z(y,x)$,
  \[
\iint_{\cX^2} z(x,y)\{f(x)-g(y)\}^2\md x \md y \ge \frac{1}{4} \iint_{\cX^2}z(x,y)[\{f(x)+g(x)\}-\{f(y)+g(y)\}]^2\md x \md y. 
  \]

\begin{proof}
\begin{align*}
  [\{f(x)+g(x)\}-\{f(y)+g(y)\}]^2&=
  [\{f(x)-g(y)\}-\{f(y)-g(x)\}]^2\\
  &\le
  [\{f(x)-g(y)\}-\{f(y)-g(x)\}]^2\\
  &~~+
[\{f(x)-g(y)\}+\{f(y)-g(x)\}]^2 \\
  &=
  2[\{f(x)-g(y)\}^2+\{f(y)-g(x)\}^2].
\end{align*}
Multiplying both sides of the inequality by the non-negative $z(x,y)$ and integrating over $\cX^2$ gives the result, since  by the symmetry of $z$,
\[
\iint_{\cX^2}2z(x,y)[\{f(x)-g(y)\}^2+\{f(y)-g(x)\}^2]\md x \md y
=
4\iint_{\cX^2}
z(x,y)\{f(x)-g(y)\}^2\md x \md y.
\]
\end{proof} 

\section{Proof of Theorem \ref{thrm.necessary}}
\label{sec.prove.Necessary}
For a stationary Markov chain $X_0,X_1,X_2,\dots$ with kernel $Q$ and a stationary distribution of $\mu$, for $h\in L_0^2(\mu)$ we define the lag-$k$ autocorrelation  to be (for any non-negative integer $i$),
\[
\rho^h_k:=\Cor{h(X_i),h(X_{i+k})}.
\]
The asymptotic variance is then \cite[e.g.,][]{Gey1992}
\begin{equation}
  \label{eqn.asympVarExpand}
\asVar_{Q}(h)=\|h\|^2_{L^2(\mu)}\left(2\sum_{k=0}^\infty \rho_k^h~-1\right). 
\end{equation}
For a reversible kernel, as we shall see, the even-lagged auto-correlations are bounded below by a multiple of the probability that the chain has not yet moved, and provided there is a left-spectral gap, the contribution from odd-numbered lags can also be bounded below. We first deal with odd-lagged auto-correlations.

\begin{prop}
  \label{prop.odd.lags}
  Let $Q$ be an ergodic, reversible Markov Kernel on $\cX$ with a stationary distribution of $\mu$. If $Q$ has a left-spectral gap of $\epsilon^L$ then for any $h\in L_0^2(\mu)$ with $\Expects{\mu}{h^2}>0$, 
  \[
\sum_{j=0}^\infty \rho^h_{2j+1}\ge -\frac{1-\epsilon^L}{2\epsilon^L-(\epsilon^L)^2}.
  \]
\end{prop}
\begin{proof}
For any $\mu$-invariant, reversible, ergodic Markov kernel, $Q$ and function $h\in L^2(\mu)$ with $\Expects{\mu}{h^2}>0$, there is \cite[e.g.][Proposition 1]{DPDK2015} a probability measure $\eta^h$ on $[-1,1)$ such that when the Markov chain is stationary, the lag-$k$ correlation is
  \[
\rho^h_k=\int_{-1}^1 \lambda^k \eta^h(\md \lambda).
  \]
  If there is a left spectral gap of $\epsilon^L$ then for all $h\in L^2(\mu)$ with $\Expects{\mu}{h^2}>0$, the support of $\eta^h$ is at most $[\epsilon^L-1,1)$. Hence
    \[
    \rho^h_{2j+1}\ge -(1-\epsilon^L)^{2j+1}.
    \]
    Thus, the sum of the odd numbered auto-correlations is
    \[
    \sum_{j=0}^\infty \rho^h_{2j+1}\ge -\sum_{j=0}^\infty (1-\epsilon^L)^{2j+1}
    =
    -\frac{1-\epsilon^L}
    {1-(1-\epsilon^L)^2},
    \]
    which gives the required result.
\end{proof}

A minimum positive contribution of the even-lagged auto-correlations is made concrete in Proposition \ref{prop.even.lags}.

\begin{prop}
  \label{prop.even.lags}
Let $Q$ be a reversible Markov kernel on $\cX$ with a stationary distribution of $\mu$. Then for any $h\in L_0^2(\mu)$ with $\Expects{\mu}{h^2}>0$,
  \[
\|h\|^2_{L^2(\mu)}\rho^h_{2j}\ge \Expects{X_0\sim \mu}{h(X_0)^2 \Prob{X_j=X_0}^2}.
  \]
\end{prop}
\begin{proof}
\begin{align*}
\rho_{2j}
=
\Expect{h(X_0)h(X_{2j})}
&=
\Expects{X_0\sim \mu,X_{-j}\sim Q^{-j}(X_0,\cdot), X_j \sim Q^j(X_0,\cdot)}{h(X_{-j})h(X_{j})}\\
\mbox{(Markov property)}&=
\Expects{X_0\sim \mu}{
  \Expects{X_{-j}\sim Q^{-j}(X_0,\cdot)}{h(X_{-j})}
\Expects{X_j \sim Q^j(X_0,\cdot)}{h(X_{j})}}\\
\mbox{(reversibility)}&=
\Expects{X_0\sim \mu}{
  \Expects{X_j\sim Q^{j}(X_0,\cdot)}{h(X_{j})}^2}\\
&\ge
\Expects{X_0\sim \mu}{
  \Expects{X_j\sim Q^{j}(X_0,\cdot)}{h(X_{j})1(X_j=X_0)}^2}\\
&=
\Expects{X_0\sim \mu}{h(X_0)^2
  \Expects{X_j\sim Q^{j}(X_0,\cdot)}{1(X_j=X_0)}^2}\\
&=
\Expects{X_0\sim \mu}{h(X_0)^2
  \Prob{X_j=X_0}^2}.
\end{align*}
\end{proof}

The lower bound on the contribution of the odd-lagged auto-correlations is itself bounded but the contribution from the even-numbered lags may be infinite (indeed, that is the motivation for the theorem), so we take some care with these terms. Combining Propositions \ref{prop.odd.lags} and \ref{prop.even.lags} through \eqref{eqn.asympVarExpand}, we obtain
\begin{align}
  \nonumber
\asVar_Q(h)
&\ge
2\lim_{k\to \infty}\Expects{X_0\sim \mu}{h(X_0)^2\sum_{j=0}^{k-1} \Prob{X_j=X_0}^{2}}
-\|h\|^2_{L^2(\mu)}-2\frac{1-\epsilon^L}{2\epsilon^L-(\epsilon^L)^2} \|h\|^2_{L^2(\mu)}\\
\label{eqn.asympVarSimpl}
&=
2\lim_{k\to \infty}\Expects{X_0\sim \mu}{h(X_0)^2\sum_{j=0}^{k-1} \Prob{X_j=X_0}^{2}}
-\frac{2-(\epsilon^L)^2}{2\epsilon^L-(\epsilon^L)^2}\|h\|^2_{L^2(\mu)}.
\end{align}

One way to achieve $(\theta_{j},w_{j})=(\theta_0,w_0)$ is to reject every proposal up to and including the $j$th, so let $\alpha_{pm}(\theta,w;\theta',w')$ be as given in \eqref{eqn.alphaPM}, and define
\begin{align*}
  \alpha_{pm}(\theta,w)&:=\Expects{(\theta',W')\sim q(\theta'|\theta)\qtil_{\theta'}(w')}{\alpha_{pm}(\theta,w;\theta',W')},\\
  S_k(\theta,w)&:=\sum_{j=0}^{k-1}\{1-\alpha_{pm}(\theta,w)\}^{2j}.
\end{align*}
For a PMMH chain, $\Prob{X_j=X_0}\equiv\Prob{(\theta_j,W_j)=(\theta_0,W_0)}\ge \{1-\alpha_{pm}(\theta,W)\}^j$. Then, since $1-\alpha\le \exp(-\alpha)$,
\[
S_k(\theta,w)
=
\frac{1-\{1-\alpha_{pm}(\theta,w)\}^{2k}}{1-\{1-\alpha_{pm}(\theta,w)\}^{2}}
\ge
\frac{1-\exp\{-2k\alpha_{pm}(\theta,w)\}}{2\alpha_{pm}(\theta,w)}.
\] 

We now create a simpler bound for $S_k$ in terms of $w$ and $r(\theta)$.
\begin{prop}
  \label{prop.monotonDecreasing}
  For $x\ge 0$, $[1-\exp(-x)]/x\ge 1/(1+x)$.
\end{prop}
\begin{proof}
  Since $\exp(x)\ge 1+x$, $\exp(-x)\le \frac{1}{1+x}$, so $1-\exp(-x)\ge x/(1+x)$.
\end{proof}
\begin{prop}
  \label{prop.alphabarbound}
  \[
  \alpha_{pm}(\theta,w)\le
  \frac{2r(\theta)}{w+r(\theta)}.
  \]
\end{prop}
\begin{proof}
  Jensen's inequality applied to $x\to 1\wedge x$  gives
\[
  \alpha_{pm}(\theta,w)\le 1\wedge \Expects{\theta'\sim q(\cdot|\theta),W'\sim \qtil_\theta(\cdot)}{\frac{W'}{w}r(\theta,\theta')}
  =
  1\wedge \frac{1}{w}r(\theta)
  \le
  2\frac{r(\theta)}{w+r(\theta)},
  \]
  where the final inequality uses $1\wedge x \le 2 x/(1+x)$. 
\end{proof}

From Proposition \ref{prop.monotonDecreasing} then Proposition \ref{prop.alphabarbound},
\[
T_k(\theta,w):=\frac{1}{k}S_k(\theta,w)\ge \frac{1}{1+2k\alpha_{pm}(\theta,w)}
\ge
\frac{1}{1+4kr(\theta)/\{w+r(\theta)\}}
=
\frac{w+r(\theta)}{w+(4k+1)r(\theta)}
\]
So
\[
S_k(\theta,w) \ge \frac{w+r(\theta)}{w/k+(4+1/k)r(\theta)}
\ge
\frac{w+r(\theta)}{1/\sqrt{k}+(4+1/k)r(\theta)}1(w\le \sqrt{k}).
\]

Thus \[
\sum_{j=0}^{k-1}\Prob{(\theta_j,W_j)=(\theta_0,W_0)}^2\ge S_k(\theta_0,W_0)\ge 
\frac{W_0+r(\theta_0)}{1/\sqrt{k}+(4+1/k)r(\theta_0)}1(W_0\le \sqrt{k}).
\]

Hence
\begin{align*}
  \lim_{k\to \infty}&
  \Expects{\theta_0,W_0\sim \pi\times \nu}
          {h(\theta_0,W_0)^2\sum_{j=0}^{k-1} \Prob{(\theta_j,W_j)=(\theta_0,W_0)}^{2}}\\
&\ge
\lim_{k\to \infty}\Expects{(\theta,W)\sim \pi\times \nu}{h(\theta,W)^2
  \frac{W+r(\theta)}{1/\sqrt{k}+(4+1/k)r(\theta)}1(W\le \sqrt{k})}\\
&\ge
\Expects{(\theta,W)\sim \pi\times \nu}{h(\theta,W)^2
  \frac{W+r(\theta)}{4r(\theta)}}\\
&=
\frac{1}{4}\|h\|^2_{L^2(\pi\times \nu)} 
+ \frac{1}{4}\Expects{(\theta,W)\sim \pi\times \nu}{h(\theta,W)^2
  \frac{W}{r(\theta)}},
\end{align*}
where the final inequality follows from Fatou's Lemma.

Substituting into \eqref{eqn.asympVarSimpl} gives
\[
\asVar_{\Ppm} \ge
\frac{1}{2}\|h\|^2_{L^2(\pi\times \nu)} 
+ \frac{1}{2}\Expects{(\theta,W)\sim \pi\times \nu}{h(\theta,W)^2
  \frac{W}{r(\theta)}}
-
\frac{2-(\epsilon^L)^2}{2\epsilon^L-(\epsilon^L)^2}\|h\|^2_{L^2(\pi\times \nu)}.
\]
However,
\[
\frac{1}{2}- \frac{2-\epsilon^2}{2\epsilon-\epsilon^2}
=
-\frac{2-\epsilon-\epsilon^2/2}{\epsilon(2-\epsilon)}
=-\frac{1}{\epsilon}+\frac{\epsilon^2/2}{\epsilon(2-\epsilon)}
=
-\frac{1}{\epsilon}+\frac{\epsilon}{4-2\epsilon}.
\]
The form presented in the statement of the theorem follows on recalling that $\nu(w)=w\qtil_{\theta}(w)$. $\square$




\section{Proofs of generalisations of Theorem \ref{theorem.varboundA}}

\subsection{Proof of Theorem \ref{theorem.varboundB}}
\label{app.proveBoundB}

We now consider a general function $h(\theta,w)\in L^2(\pi\times \nu)$. As in \eqref{eqn.define.gw} this function engenders and infinity of functions $h_w:\cX\to \mathbb{R}$, indexed by $w\in \cW$: $h_w(\theta)=h(\theta,w)$.

We require the following:

\begin{prop}
  \label{prop.usePD}
  For $g_w,g_{w'}\in L^2(\pi)$,
  \[
  \int_{\cW^2}\qtil(w)\qtil(w') \{w\wedge w'\} \langle g_{w},g_{w'}\rangle_{L^2(\pi)} \md w \md w'\ge 0.
  \]
\end{prop}
\emph{Proof}: Denote the integral by $I$. Then because $w\wedge w'=\int_0^\infty 1(w\le z)1(w'\le z)\md z$,
\begin{align*}
  I&=\int_{\cW^2}\int_{0}^\infty\int_{\cX}
  \pi(\theta)\qtil(w)\qtil(w') 1(w\le z)1(w'\le z)
  g_{w}(\theta)g_{w'}(\theta) \md w \md w' \md \theta \md z\\
  &=
  \int_{0}^\infty\int_{\cX}\pi(\theta)
  \left\{\int_{\cW}\qtil(w)1(w\le z)g_w(\theta) \md w\right\}^2\md \theta \md z\\
  &\ge 0.~\square
\end{align*}
    
From \eqref{eqn.DirichletBoundMain} and Proposition \ref{prop.usePD},
\begin{align*}
  \cE_{\Ppm}(g)
  &\ge
  \frac{\epsMH}{4} \int_{\cW\times \cW}\qtil(w)\qtil(w')
  \{w\wedge w'\}~\left\{\|g_w\|^2_{L^2(\pi)}+\|g_{w'}\|^2_{L^2(\pi)}\right\} \md w \md w'.
\end{align*}

Replicating the derivation in Stage 2 of the proof of Theorem \ref{theorem.varboundA} we must take extra care because $h_w(\theta)\ne h_{w'}(\theta)$.
\begin{align*}
  \langle h, g\rangle_{L^2(\pi\times \nu)}
  &=
  \int \pi(\theta)w\qtil(w)  h_w(\theta) g_w(\theta) \md \theta \md w\\
  &=
  \int w\qtil(w) \langle h_w, g_w\rangle_{L^2(\pi)} \md w\\
  &=
  \int w\qtil(w) w'\qtil(w')\langle h_w, g_w\rangle_{L^2(\pi)} \md w\md w'\\
  &=
  \frac{1}{2}
 \int w\qtil(w) w'\qtil(w')\left\{\langle h_w, g_w\rangle_{L^2(\pi)}+\langle h_{w'},g_{w'}\rangle_{L^2(\pi)}\right\} \md w\md w', 
\end{align*}
by relabelling $w\leftrightarrow w'$ and averaging.

The Cauchy-Schwarz inequality then leads to the bound
\[
\langle h, g\rangle_{L^2(\pi\times \nu)}
\le
\frac{1}{2}\int \qtil(w) \qtil(w')w w'
\left\{ \|h_w\|_{L^2(\pi)} \|g_w\|_{L^2(\pi)}+\|h_w\|_{L^2(\pi)} \|g_w\|_{L^2(\pi)}\right\} \md w \md w'.
\]
We are in a position to prove Theorem \ref{theorem.varboundB} via the variational form \eqref{eqn.varRepvar}. Notice that $4\langle h,g\rangle_{L^2(\pi\times \nu)}-2\cE_{\PMH}(g)$ is the sum of two terms, the first of which is
\[
\int \qtil(w)\qtil(w')
\left\{
2 \|h_w\|_{L^2(\pi)}\|g_w\|_{L^2(\pi)}ww' - \frac{1}{2}\epsMH(w\wedge w')\|g_w\|^2_{L^2(\pi)}
\right\}\md w \md w'. 
\]
The second is the same as the first but with $w\leftrightarrow w'$. Again noting that $ax-cx^2\le a^2/(4c)$, the first term is no greater than
\[
\frac{2}{\epsMH}
\int \qtil(w)\qtil(w') \frac{w^2(w')^2\|h_w\|_{L^2(\pi)}^2}{w\wedge w'} \md w \md w'
=
\frac{2}{\epsMH}
\int \qtil(w)\qtil(w') w w' (w \vee w') \|h_w\|_{L^2(\pi)}^2 \md w \md w'.
\]
Thus
\begin{align*}
  \asVar_{\Ppm}(h)+\|h\|^2_{L^2(\pi\times \nu)}
&=4\langle h,g\rangle_{L^2(\pi\times \nu)}-2\cE_{\PMH}(g)\\
&\le
\frac{2}{\epsMH}
\int \qtil(w)\qtil(w') w w' (w \vee w') \{\|h_w\|_{L^2(\pi)}^2+\|h_{w'}\|_{L^2(\pi)}^2\}\md w \md w',
\end{align*}
as required. $\square$

\subsection{Proof of Theorem \ref{theorem.varboundCorrPM}}
Firstly, the stationary density of the correlated pseudo-marginal chain is $\nu(w)=w\qtil(w)$, the same as that or the uncorrelated chain, since
\[
\pi(\theta)q(\theta'|\theta) w\qtil(w)\qtil(w'|w) \left\{1\wedge \frac{w'\pi(\theta')q(\theta|\theta')}{w \pi(\theta) q(\theta'|\theta)}\right\}
\]
is invariant to $(\theta,w)\leftrightarrow(\theta',w')$ because $\qtil(w)\qtil(w'|w)=\qtil(w,w')=\qtil(w',w)$. Thus \eqref{eqn.defineEPtil} becomes
\[
\cE_{\Ptil}(g)\ge \cE_{\Ptilpm}(g)=\frac{1}{2}\int \qtil(w,w')(w\wedge w') D(w,w') \md w \md w'
\]
with the same definition of $D(w,w')$. Hence, \eqref{eqn.DirichletBoundMain} becomes
\begin{equation}
  \label{eqn.DirichletBoundCorr}
  \cE_{pm}\ge
  \frac{1}{4}\epsMH\int \qtil(w,w') \{w\wedge w'\} \|g_w+g_{w'}\|^2_{L^2(\pi)}\md w \md w'.
\end{equation}

Next, consider the quantity
\[
p_*(w,w'):= w^{1/2}b(w) \qtil(w,w') b(w'){w'}^{1/2}.
\]
From \eqref{eqn.corrPMcond},
\begin{align*}
  \int p_*(w,w') \md w'
  &= w^{1/2}b(w) \qtil(w)\int {w'}^{1/2}b(w) \qtil(w'|w) \md w'\\
&=
w^{1/2}b(w) \qtil(w) \Expect{b(W){W'}^{1/2}|W=w}\\
&\ge
c_{b}w^{1/2}b(w) \qtil(w) w^{1/2}/b(w)\\
&=c_b w\qtil(w)=c_b \nu(w).
\end{align*}

Using this inequality
\begin{align*}
  \langle h_*,g\rangle_{L^2(\pi\times \nu)}
  &=
\int_\cW w\qtil(w) \langle h_*, g_w\rangle_{L^2(\pi)} \md w
\le
c_b^{-1}\iint_{\cW^2} p_*(w,w')\langle h_*, g_w\rangle_{L^2(\pi)} \md w \md w'\\
&=
c_b^{-1}\iint_{\cW^2} p_*(w,w')\langle h_*, g_{w'}\rangle_{L^2(\pi)} \md w \md w',
\end{align*}
where the final equality follows from the symmetry of $p_*$. Hence
\begin{align}
  \nonumber
4\langle h_*, g\rangle_{L^2(\pi\times \nu)}
&\le
2c_{b}^{-1}\iint_{\cW^2} p_*(w,w')\langle h_*, g_w+g_{w'}\rangle_{L^2(\pi)} \md w \md w'\\
\nonumber
&\le
\iint p_*(w,w') 2c_b^{-1}\|h_*\|_{L^2(\pi)} \|g_w+g_{w'}\|_{L^2(\pi)} \md w \md w'\\
&=
\iint \qtil(w,w') 2c_b^{-1}\|h_*\|_{L^2(\pi)} w^{1/2}b(w)b(w'){w'}^{1/2}\|g_w+g_{w'}\|_{L^2(\pi)} \md w \md w'.
\label{eqn.innerProdBoundCorr}
\end{align}

Combining \eqref{eqn.DirichletBoundCorr} and \eqref{eqn.innerProdBoundCorr} through \eqref{eqn.varRepvar}, $\asVar_{\Ppm}(h_*) +\|h_*\|^2$ is
\begin{align*}
&\le
\int\qtil(w,w')
\left\{2 c_{b}^{-1}w^{1/2}b(w)b(w') {w'}^{1/2} \|h_*\|_{L^2(\pi)} \|g_w+g_{w'}\|_{L^2(\pi)}
-\frac{1}{2} (w\wedge w')\|g_w+g_{w'}\|^2_{L^2(\pi)}\right\}\md w \md w'\\
&\le
\int\qtil(w,w') \frac{4c_{b}^{-2}\|h_*\|^2_{L^2(\pi)}wb(w)^2b(w')^2w'}{2(w\wedge w')} \md w \md w'\\
&=
2\|h_*\|^2_{L^2(\pi)}c_{\gamma}^{-2}\int_{\cW^2}\qtil(w,w') b(w)^2b(w')^2(w \vee w') \md w \md w',
\end{align*}
since $w w' = (w\vee w')(w\wedge w')$. 
The final inequality follows as when $c> 0$, $ax-cx^2=-c[x-a/(2c)]^2+a^2/(4c) \le a^2/4c$. $\square$

\subsection{Proof of Proposition \ref{prop.corrPMheavy}}
\label{sec.proveCorrHeavy}
  With the formulation in \eqref{eqn.heavyCorr} for $\eta>0$,
  \begin{align*}
  W' &= (w+1)^{\cos^2 U} \exp\left(\frac{Z^2}{2a}\right)-1
  =
  (w+1)^{\cos^2 U} \left\{\exp\left(\frac{Z^2}{2a}\right)-1\right\}+(w+1)^{\cos^2 U}-1\\
  &\ge w^{\cos^2 U} \left\{\exp\left(\frac{Z^2}{2a}\right)-1\right\}.
  \end{align*}
  So, by Jensen's inequality
  \[
  \Expect{{W'}^\eta|W=w} \ge w^{\eta {\Expect{\cos^2 U}}}\Expect{\left\{\exp\left(\frac{Z^2}{2a}\right)-1\right\}^\eta}
  =
  c_{\eta}w^{\eta/2},
  \]
  where $c_\eta=\Expect{\left\{\exp\left(\frac{Z^2}{2a}\right)-1\right\}^\eta}>0$.
Setting $\eta=2/3$, we see $\qtil(w,w')$ satisfies Assumption \ref{eqn.assump.corrPM} with $b(w)=w^{1/6}$ as required.    
 
\section{Proofs of CLT-based and propositions}
  
\subsection{Proof of Proposition \ref{prop.PMCLT}}
\label{sec.prove.PMCLT}
  Under Assumption \ref{ass.lognormalCLT}, $\log W \sim \mathsf{N}(-\sigma^2/2,\sigma^2)$ and $\log W' \sim \mathsf{N}(-\sigma^2/2,\sigma^2)$ are independent. Hence
  \[
\log(WW')\sim \mathsf{N}(-\sigma^2,2\sigma^2)
~~~\mbox{and}~~~
\log(W'/W)\sim \mathsf{N}(0,2\sigma^2)
\]
are also independent. We write $\log(WW')=-\sigma^2+\sigma\sqrt{2}Z'$ and
$\log(W'/W)=\sigma\sqrt{2}Z$ where $Z$ and $Z'$ are independent standard Gaussians. Then
\begin{align*}
  R(\sigma)
  &=
  \Expect{WW'(W\vee W')}
  =
  \Expect{(WW')^{3/2} \left( \frac{W}{W'}\vee \frac{W'}{W} \right)^{1/2}}\\
  &=
  \exp\left(-\frac{3}{2}\sigma^2\right)
  \Expect{\exp\left(\frac{3\sigma \sqrt{2}}{2}Z'\right)}
  \Expect{\exp\left(\frac{1}{2}\sigma\sqrt{2}\{(-Z)\vee Z\}\right)}\\
  &=
  \exp\left(\frac{3}{4}\sigma^2\right)
  \Expect{\exp\left(\frac{1}{2}\sigma\sqrt{2}~|Z|\right)}\\
  &=
 2 \exp\left(\frac{3}{4}\sigma^2\right)
 \int_0^\infty\exp\left(-\frac{1}{2}z^2+\frac{1}{2}\sigma\sqrt{2}z\right)\\
 &=
 2 \exp\left(\sigma^2\right)
 \int_0^\infty\exp\left\{-\frac{1}{2}\left(z-\frac{\sigma }{\sqrt{2}}\right)^2\right\}\\
 &=
 2 \exp\left(\sigma^2\right)
\Phi\left(\frac{\sigma}{\sqrt{2}}\right). 
  \end{align*}



  \subsection{Proof of Proposition \ref{prop.corrPM.CLT}}
\label{sec.prove.corrPMCLT}
  Given the bivariate Gaussian form,
  $\log W' |(W=w) \sim \mathsf{N}(\mu',{\sigma'}^2)$,
  with $\mu'=-\sigma^2/2 + \rho(\log w+\sigma^2/2)$ and ${\sigma'}^2=\sigma^2(1-\rho^2)$. We consider $b(w)=w^{1/2-\gamma}$, so we must prove that $\Expects{\qtil}{{W'}^{1-\gamma}|W=w}\ge w^{\gamma}$.

  First, $(1-\gamma) \log W' \sim \mathsf{N}(\{1-\gamma\}\mu',\{1-\gamma\}^2{\sigma'}^2)$ and hence
  \begin{align*}
    \Expect{{W'}^{1-\gamma}|W=w}
    &=
    \exp\left[(1-\gamma)\mu'+\frac{1}{2}(1-\gamma)^2 {\sigma'}^2\right]\\
    &=
    w^{\rho(1-\gamma)}\exp\left[
      -\frac{1}{2}\sigma^2(1-\gamma)(1-\rho)+\frac{1}{2}\sigma^2(1-\gamma)^2(1-\rho^2)
      \right].
    \end{align*}
    Requiring $\rho(1-\gamma)=\gamma$ fixes $\gamma=\rho/(1+\rho)$, which also means that $(1-\gamma)(1+\rho)=1$, so that the terms in the exponential sum to $0$.

\end{document}